\documentclass[reqno,11pt]{amsart} 

\usepackage[colorlinks=true,urlcolor=blue,
citecolor=red,linkcolor=blue,linktocpage,pdfpagelabels,
bookmarksnumbered,bookmarksopen]{hyperref}

\usepackage[left=2.6cm,right=2.6cm,top=2.9cm,bottom=2.9cm]{geometry}

\usepackage{amsmath,amssymb,latexsym,soul,cite,mathrsfs}
\numberwithin{equation}{section}  
\usepackage{color,enumitem,graphicx}

\usepackage{cancel}

\usepackage[colorlinks=true,urlcolor=blue,
citecolor=red,linkcolor=blue,linktocpage,pdfpagelabels,
bookmarksnumbered,bookmarksopen]{hyperref}
\usepackage[english]{babel}

\newtheorem{theorem}{Theorem}[section]
\newtheorem{proposition}[theorem]{Proposition}
\newtheorem{coring}[theorem]{Corollary}
\newtheorem{lemma}[theorem]{Lemma}
\DeclareMathAlphabet{\mathpzc}{OT1}{pzc}{m}{it}
\newtheorem{definition}[theorem]{Definition}
\newtheorem{example}[theorem]{Example}

\newtheorem{remark}[theorem]{Remark}

\newtheorem{theoremletter}{Theorem}

\title[On the $\sigma_2$-curvature and Volume of compact manifolds]{On the $\sigma_2$-curvature and Volume of compact manifolds}

\author[M. Andrade]{Maria Andrade}
\address[M. Andrade]{Department of Mathematics, 
	Federal University of Sergipe
	\newline\indent 
	49100-000, Sao Cristov\~ao-SE, Brazil}
\email{\href{mailto: maria@mat.ufs.br}{ maria@mat.ufs.br}}

\author[T. Cruz]{Tiarlos Cruz}

\address[T. Cruz]{ Institute of Mathematics, 
	Federal University of Alagoas
	\newline\indent 
	57072-970, Maceió-AL, Brazil}
\email{\href{mailto: cicero.cruz@im.ufal.br}{cicero.cruz@im.ufal.br}}

\author[A. Silva Santos]{Almir Silva Santos}

\address[A. Silva Santos]{Department of Mathematics, 
	Federal University of Sergipe
	\newline\indent 
	49100-000, Sao Cristov\~ao-SE, Brazil}
\email{\href{mailto: almir@mat.ufs.br}{ almir@mat.ufs.br}}

\thanks{The authors thank the anonymous referee for the valuable suggestions and comments. The first and third authors were partially supported by Brazilian National Council for Scientific and Technological Development (CNPq Grant 403349/2021-4) and FAPITEC/SE/Brazil. The second author was partially supported by the Brazilian National Council for Scientific and Technological Development  (CNPq Grant 405468/2021-0 and 311803/2019-9)}

\begin{document}

\maketitle
\begin{abstract}
 In this work we are interested in studying deformations of the $\sigma_2$-curvature and the volume.  For closed manifolds, we relate critical points of the total $\sigma_2$-curvature functional to the $\sigma_2$-Einstein metrics and, as a consequence  of  results  of  M. J. Gursky and J. A. Viaclovsky \cite{MR1872547} and Z. Hu and H. Li \cite{MR2052939},  we obtain  a sufficient and necessary condition  for a critical metric to be Einstein. Moreover,  we show a volume comparison result for Einstein manifolds with respect to $\sigma_2$-curvature  which  shows that the volume can be controlled by the $\sigma_2$-curvature under certain conditions. Next, for compact manifold with nonempty boundary, we study variational properties of the volume functional restricted to the space of metrics with constant $\sigma_2$-curvature and with fixed induced metric on the boundary.  We  characterize the critical points to this functional as the solutions of an equation and show that  in space forms they are  geodesic balls. Studying  second order properties of the volume functional we show that there is a variation for which geodesic balls are indeed local minima in a natural direction. 
\end{abstract}


\section{Introduction}

Given a compact smooth Riemannian manifold $(M^{n},g)$ of dimension $n\geq 3$, there exists an orthogonal decomposition of the curvature tensor $Rm_g$ which is given by $Rm_g=W_g+A_g\odot g$, where $\odot$ is the Kulkarni-Nomizu product, $W_g$ is the Weyl tensor and $A_g$ is the Schouten tensor defined as
\begin{equation}\label{eq050}
    A_g=\frac{1}{n-2}\left(Ric_g-\frac{R_g}{2(n-1)}g\right).
\end{equation}
Here $Ric_g$ and $R_g$ are the Ricci and scalar curvature of
the metric g, respectively, (e.g., see \cite{MR2274812,besse}). For $k\in\{1,\ldots,n\}$, the $\sigma_k$-curvature is defined as 
$$\sigma_k(g):=\sum_{1\leq i_1<\cdots< i_k\leq n}\lambda_{i_1}\cdot\ldots\cdot\lambda_{i_k},$$
where $\lambda_1,\ldots,\lambda_{n}$ are the eigenvalues of $A_g$; that is, $\sigma_k(g)$ is the $k$-th elementary symmetric function of the eigenvalues of $A_g$.

The $\sigma_k$ equation is always elliptic for $k=1$. The $\sigma_1$-curvature is the scalar curvature, up to a dimensional constant, and it was extensively studied along the years. For $k\geq 2$ the picture is quite different. In fact, to assure the ellipticity of the equation we need some additional assumption. For example, a sufficient condition for this is that $g$ is {\it positive} or {\it negative} {\it $k$-admissible}. By definition, a metric $g$ on $M$ is said to be positive $k$-admissible if it belongs to the $k$-th cone $\Gamma_k^+$; this means that a metric $g$ belongs to $\Gamma_k^+$ if and only if $\sigma_i(g)>0$ for all $i=1,\ldots,k$. The negative cone $\Gamma_k^-$ is defined similarly. See \cite[Section 6]{MR1738176} and \cite[Section 3]{MR1764770}.

In the last few decades much attention has been played to the study of the $\sigma_2$-curvature, see for example \cite{MR2104700, MR3828913, MR2015262, MR3663325, MR1738176, MR1872547} and the references therein. 

We are interested to study conformal and non conformal deformations of the $\sigma_2$-curvature. We regard the $\sigma_2$-curvature as a nonlinear map on the space $\mathcal M$ of all Riemannian metrics on $M$,
$$
\sigma_2: \mathcal{M} \rightarrow C^{\infty}(M);\quad g \mapsto \sigma_2(g).
$$
See \eqref{eq034} for an explicit expression of the $\sigma_2$-curvature in terms of the Ricci and scalar curvature.

In order to understand the behavior of this map  we  consider its linearization at a metric $g$ given by  $\Lambda_{g}: S_{2}(M) \rightarrow C^{\infty}(M)$ and also its $L^{2}$-formal adjoint, denoted by $\Lambda_{g}^{*}: C^{\infty}(M) \rightarrow S_{2}(M)$. Here $S_{2}(M)$ is the space of symmetric 2-tensors on $M$. It was proved in \cite{MR3828913} that $\Lambda_g^*(1)$ give us a 2-tenson canonically  associated with the $\sigma_2$-curvature (see Section \ref{sec-background}). In this way we say that a metric $g$ is $\sigma_2$-{\it Einstein} if $\Lambda_g^*(1)=\kappa g$ for some constant $\kappa$. 

J. A. Viaclovsky \cite{MR1738176} has considered the following  functional, defined in $\mathcal M$,
\begin{equation}\label{functionalv}
\mathcal F_k(g)=\int_M\sigma_k(g)dv_g.
\end{equation}
He has proved that for $k\not=n/2$ and for $(M,g)$ locally conformally flat in case $k\geq 3$, then the Euler-Langrange equation of $\mathcal F_k$ restricted to a conformal class of a metric $g$ is $\sigma_k(g)=$ constant. Later S. Brendle and J. A. Viaclovsky \cite{MR2071927} extended this result for $k=n/2$. Also, it is well known that  critical metrics of $\mathcal F_1$ restricted to the space  of unit volume metrics are exactly Einstein metrics \cite[Theorem 4.21]{besse}. For $k>1,$ the problem is more intriguing since the   Euler-Lagrange equations are  fourth order equations in the metric. In fact,  T. P. Branson and A. R. Gover \cite{MR2389992} have studied the functional $\mathcal F_k$ in a conformal class. While the equation $\sigma_1(g)=$ constant is always variational, they have shown that for $k\in\{3,\ldots,n\}$ the equation $\sigma_k(g)=$ constant is variational if and only if the manifold is locally conformally flat. However, in dimension three it holds an interesting result proved by M. J. Gursky and J. A. Viaclovsky \cite{MR1872547}.
\begin{theorem}[M. J. Gursky, J. A. Viaclovsky]\label{gursky-viaclovsky}
Let  $(M^3,g)$ be a closed manifold of dimension three. Then a metric $g$ with $\mathcal F_2(g) \geq 0$ is critical for $\mathcal F_2$ restricted to the space of unit volume metrics if and only if $g$ is Einstein and therefore  has constant sectional curvature. 
\end{theorem}

They observed that the condition $\mathcal F_2(g)\geq 0$ is necessary, since we can write the $\sigma_2$-curvature  as $\sigma_2(g)=-\frac{1}{2}|\mathring{Ric}_g|^2+\frac{1}{48}R_g^2,$
where $\mathring{Ric}_g=Ric-\frac{1}{3}R_gg$ denotes the trace free Ricci tensor. Also, they have proved that there exists a Berger metric on $\mathbb S^3$ which is a critical metric with $\mathcal F_2(g)<0$. Later, Z. Hu and H. Li \cite{MR2052939} generalized this result for $n > 4$, in the case that the  metric is locally conformally flat. G. Catino, P. Mastrolia and D. D. Monticelli \cite{MR3478936} have extended this result to the noncompact setting. They obtained that a complete critical metric for $\mathcal F_2$ with non-negative scalar curvature is flat. The non-negativity condition on the scalar curvature cannot be removed.

We remark that  Theorem \ref{gursky-viaclovsky} cannot hold in four-dimensional manifolds on account of the specificity of this dimension. Indeed, in this dimension we have the\textit{ Gauss-Bonnet-Chern Formula}
\begin{equation}\label{Gauss-Bonnet-int}
    \int_M \left(\sigma_2(g)+\frac{|W_g|^2}{4}\right)
\;dv_g=8\pi^2\chi(M),
\end{equation}
where $\chi(M)$ denotes the Euler characteristic of $M$. This formula  serves as a ``bridge" between the topological and geometric information, which implies that under conformal change of metrics, since $|W_g|^2dv_g$ is point-wisely conformally invariant, the integral $\int_M\sigma_2(g)dv_g$ is conformally invariant.

As a consequence of the results by M. J. Gursky and J. A. Viaclovsky \cite{MR1872547} and Z. Hu and H. Li \cite{MR2052939} we obtain our first result, which reads as follows.

\begin{theoremletter}\label{theorem-A}
Let $(M^{n},g)$ be a closed Riemannian manifold of dimension $n\geq 3$, $n\not=4$, which is locally conformally flat for $n\geq 5$. Then $(M^n,g)$ is a $\sigma_2$-Einstein manifold with $\sigma_2(g)\geq 0$ if and only if $(M^n,g)$ is an Einstein manifold.
\end{theoremletter}

We will see  that $\sigma_2$-Einstein metrics are {\it critical metrics} of the volume functional, justifying our interest. Moreover, we present applications of Theorem \ref{theorem-A} in the second variation formula of Viaclovsky's functional  \eqref{functionalv} for $k=2$. To this end, we compute the formula for the second derivative of the
$\sigma_2$-curvature (Proposition \ref{lem004}). It should be mentioned that  this formula is of independent interest and we hope that can be useful in other contexts.

We recall that a closed Einstein manifold $(M^n, g)$ with dimension $n\geq 3$ is said to be stable if the {\it Einstein operator}
\begin{equation}\label{eq004}
    \Delta_{E}^g:=\nabla^{*} \nabla-2 \mathring{R}: \Gamma\left(S^{2} M\right) \rightarrow \Gamma\left(S^{2} M\right)
\end{equation}
 restricted to $S_{2,g}^{TT}(M):=\{h \in S_2(M): \delta_g h = 0, tr_gh  = 0\}$
 is nonnegative, i.e, if there exists $\lambda\geq 0$ such that $\left\langle\Delta_E^gh,h \right\rangle\geq\lambda\|h\|^2,$
 for any $h\in S_{2,g}^{TT}(M)$. If $\lambda>0$, then the metric is said {\it strictly stable}, see \cite[Definition 4.63]{besse}. Otherwise, $g$ is said {\it unstable}. Here $\mathring R(h)_{ij}=g^{kl}g^{st}R_{kijs}h_{lt}$ for any $h\in S_2(M)$. Stability of Einstein manifolds plays a fundamental role in the study of Einstein manifolds, see for instance \cite{besse} and references therein. The space $S_{2,g}^{TT}(M)$ is often called the space of transverse-traceless tensors, or TT tensors for short.

 We present a geometric characterization of $\sigma_2$-curvature showing a volume comparison theorem  for metrics close to strictly stable  Einstein metrics. It should be mentioned that a similar question was first addressed in the context of scalar curvature and $Q$-curvature by W. Yuan\cite{yuan2016volume} and Y. Lin and W. Yuan in \cite{lin2021deformations}, respectively.

\begin{theoremletter}\label{theorem-B}
For $n\geq 3$, suppose $(M^{n},g_0)$ is an $n$-dimensional closed strictly stable Einstein manifold with Ricci curvature
$
Ric_{g_0} = (n-1)\lambda g_0,$
where $\lambda >0$ is a constant. 
Then there exists a constant $\varepsilon_0 > 0$ such that for  any metric $g$ on $M$ satisfying
$$\sigma_2(g)\geq \sigma_2(g_0)\qquad\mbox{ and }\qquad\|g - g_0\|_{C^2}<\varepsilon_0,$$
we have the following volume comparison 
$$V(g) \leq V(g_0),$$
with the equality holding if and only if $g$ is isometric to $g_0$. 
\end{theoremletter}

Since the round sphere is strictly stable (see for instance \cite[Section 3.1]{Kroncke}), we obtain the following immediate result.

\begin{coring}\label{volsi}
	Let $(\mathbb{S}^{n},g_0)$ be the round sphere, with $n\geq 3$. Then there exists a constant $\varepsilon_0 > 0$ such that for any metric $g$ on $\mathbb{S}^n$ satisfying
	$$\sigma_2(g)\geq \frac{n(n-1)}{8}\qquad\mbox{ and }\qquad\|g - g_0\|_{C^2}<\varepsilon_0,$$
we have 
	$$V(\mathbb S^n,g) \leq V(\mathbb{S}^{n},g_0),$$ with the equality holding if and only if $g$ is isometric to $g_0$.
\end{coring}

After we finished this paper, we learned that Y. Fang, Y. He and J. Zhong \cite{Fang-He-Zhong} proved independently the same comparison theorem. They found examples showing that the strictly stable assumption cannot be dropped in Theorem \ref{theorem-B}. Also they studied the Ricci-flat case.

Motivated by the characterization of Einstein metrics as critical points of the volume functional restricted to the space of metrics with constant scalar curvature $-1$ on a closed manifold, P. Miao and L.-F. Tam \cite{MR2546025} have shown necessary and sufficient conditions to a metric be a critical point of the volume functional restricted to the space of all Riemannian metric with constant scalar curvature and, if the manifold possesses a boundary, the   induced metric on the boundary is fixed.  
The condition is related with the $L^2$-formal adjoint of the linearization of the scalar curvature. In fact, in the literature these critical points are the so-called  $V$-{\it static} metrics. 
Such a  concept was introduced by  P. Miao and L.-F. Tam  \cite{MR2546025}  and are  useful as an attempt to better understand the interplay between scalar curvature and volume. Later, in \cite{miao2011einstein}, they classified all Einstein or conformally flat metrics which are critical points for the volume functional restricted to the above space. See also the works \cite{miao2011einstein,MR2546025,MR3096517,baltazar2017critical,MR3427144} and references therein.

Before we state our next result, let us define the following self-adjoint operator in the $L^2$-norm $\mathcal T_g:C_0^\infty(M)\rightarrow C_0^\infty(M)$, which is related with the conformal change of the $\sigma_2$-curvature \eqref{eq047}, as
\begin{equation}\label{eq049}
\mathcal T_g(u):=\frac{1}{2}\langle T_1,\nabla^2u\rangle+2u\sigma_2(g)=\frac{1}{2}div\left(T_1(\nabla u)\right)+2u\sigma_2(g),
\end{equation}
where $C^\infty_0(M)$ is the space of smooth functions defined in $M$ which are equal to zero on the boundary, $T_1$ is the first Newton transformation associated with $A_g$ and it is given by
\begin{equation}\label{eq018}
T_1=\frac{1}{n-2}\left(\frac{1}{2}R_gg-Ric_g\right).    
\end{equation}
We notice that by the contracted second Bianchi identity, $T_1$ is divergence free. 
It is well known that if $g$ is a positive 2-admissible metric, i.e., it belongs to the positive elliptic convex 2-cone, 
$\Gamma_2^+:=\{g\in\mathcal M;\sigma_1(g)>0\mbox{ and }\sigma_2(g)>0\}$,
then $T_1$ is positive definite (See \cite{MR806416,MR0113978} and \cite[Proposition 2.1]{MR2015262}). This implies that the operator $\mathcal T_g$ is elliptic.  Moreover, if $g\in \overline{\Gamma_2^+}$, then $T_1\geq 0$. 
 
Let $\mathcal{M}^{K}$ be the space of metrics with constant $\sigma_2$-curvature $K\neq 0$. The critical points of the volume functional restricted to $\mathcal{M}^{K}$ are precisely stationary points of $\mathcal{F}_2$ (see \eqref{functionalv}) restricted to $\mathcal{M}^{K}$.  It was proved in  \cite[Theorem 6.2]{case2019conformally}, under some assumptions,  that a metric $g$ is critical point of the volume functional restricted to the space of metrics with constant $\sigma_2$-curvature if and only if there is a function $f\in C^\infty(M)$ such that $\Lambda^*_g(f)=g$. More generally,  J. S. Case, Y.-J. Lin and W. Yuan \cite{case2019conformally} have defined a {\it conformally variational invariant} (CVI) as a conformally Riemannian scalar invariant which is homogeneous and has a conformal primitive. One recall that CVIs generalizes concepts as scalar curvature, Branson’s Q-curvature and $ \sigma_2 $-curvature. They extended to the CVI context, for closed manifold, some well known results for scalar curvature (see also \cite{Case_Lin_Yuan_2020}).

 Inspired by the above discussion and the V-static works, we also establish a similar result for the $\sigma_2$-curvature on compact manifolds with boundary whose proof contrasts sharply with the closed case due to the nature of the issue.  We investigate critical points of the volume functional on the space  of metrics  which have constant $\sigma_2$-curvature $K$ and fixed metric $\gamma$ on the boundary, i.e, 
$$
 \mathcal M_{\gamma}^K=\{g\in \mathcal M: \sigma_2(g)=K\quad\mbox{and}\quad g|_{T\partial M}=\gamma\}.
$$ 
In addition,  we show that these critical  points give a class of manifolds with very nice properties.

\begin{theoremletter}\label{theorem-C}
Let $(M^n,g)$ be a compact Riemannian manifold of dimension $n\geq 3$ with nonempty boundary. Let $g$ be a 2-admissible metric such that  the first Dirichlet eigenvalue of $-\mathcal T_g$ is positive. Then, $g$ is a critical point of the volume functional $V:{\mathcal{M}_{\gamma}^K}\rightarrow\mathbb R$ if and only if there exists a smooth function $f$ on $M$ such that $f  =  0$  on $\partial M$ and 
\begin{equation}\label{eq0100}
\Lambda_g^*(f) =  g  \text{ in }M.
\end{equation}
\end{theoremletter}

We can see Theorem \ref{theorem-C} as a counterpart to the manifold with boundary, at least for the $\sigma_2$-curvature. In fact, as expected in \cite[Remark 6.3]{case2019conformally}, when $f|_{\partial M}$ is not necessarily zero,  we obtain the Proposition \ref{VH}, which gives us an interesting relation involving the boundary curvature associated to the $\sigma_2$-curvature defined in \cite{MR2570314}.

Metrics satisfying equation \eqref{eq0100} are special, because they have many interesting properties. For instance, they have constant $\sigma_2$-curvature and  critical metrics in space forms  are  geodesic balls (see Theorem \ref{classi}). More generally, in Einstein manifolds (we will see in Section \ref{space form}) the  equation \eqref{eq0100} is closely related with 
$V$-static metric (or Miao-Tam critical metric as first denoted in \cite{barros2015bach}). For more results  see \cite{miao2011einstein,barros2015bach,yuan2016volume,batista2017critical,baltazar2017critical}.

In this paper we also study second order properties for the volume  functional on $\mathcal{M}_\gamma^K$ and we show that there is a variation for which geodesic balls are indeed local minima for the volume functional in a natural direction. We also observe that it is not difficult modify our proof for a manifold without boundary. We point out that such result  is new even in the closed case.

\begin{theoremletter}\label{theorem-D}
Let $\Omega$ be a geodesic ball  with compact closure in $\mathbb{S}^n_+$ and $\mathbb{H}^n$, and let $g$ be the standard metric on $\Omega$.   Let $\gamma=g|_{T\partial \Omega}$ and $K$ be the constant equal the $\sigma_2$-curvature of $g.$ 
\begin{itemize}
\item[(a)]If $\Omega\subset\mathbb{S}_+^{n}$,  there exists a smooth path  $\{g(t)\}$  in $\mathcal M^K_{
\gamma}$ such that $V'(0)=0$ and $V''(0)>0$. 
\item[(b)]  If $\Omega\subset \mathbb{H}^n,$ given $p\in M$ there exists a geodesic ball centered at $p$  with radius $\delta>0$, which depends on $p$ and $(\Omega,g)$,  and a smooth path  $\{g(t)\}$  in $\mathcal M^K_{
\gamma}$ such that $V'(0)=0$ and $V''(0)>0$. 
\end{itemize}
Here $V(t)$ is the volume of $(\Omega,g(t))$. In particular, for the above models the volume of the standard metric $g$ are strict local minimum along those  variations.
\end{theoremletter}

The organization of the paper is as follows. 
In Section \ref{sec-background},  we give some preliminaries and an overview about  deformations of $\sigma_2$-curvature. In Section \ref{closed section}, we study critical metrics in closed manifolds and prove Theorem \ref{theorem-A}. Moreover, we also study second order properties of the $\sigma_2$-curvature  in order to compute the second variation of the functional \eqref{functionalv} for $k=2$. In Section \ref{volume comparison}, we find some variational formulae in order to prove Theorem \ref{theorem-B}.
We study variational properties of the volume functional, constraint to the space of metrics of constant $\sigma_2$-curvature with a prescribed boundary metric in Section \ref{sec003}, moreover we prove Theorem \ref{theorem-C}.
In Section \ref{space form}, we give  some examples of functions satisfying  \eqref{eq0100}, we also show that the only domains in the space forms $\mathbb{H}^{n}$ or $\mathbb{S}^{n}$, on which the canonical metrics are critical points, are geodesic balls. In Section \ref{secondvariat},  we  compute the second variation of the volume functional at critical points in $\mathcal M_\gamma^K$. Then we calculate the second variation formula for the volume functional in order to prove Theorem \ref{theorem-D}.

\section{Background}\label{sec-background}
Let $(M^{n},g)$ be a Riemannian manifold of dimension $n\geq 3$, with or without boundary. The $\sigma_k$-{\it curvature}, which we denote by $\sigma_k(g)$, is defined as the second elementary symmetric function of the eigenvalue of the {\it Schouten tensor}, see \eqref{eq050} for its definition. Since $\sigma_1(g)  =  tr_g A_g$ and $\sigma_2(g)=\frac{1}{2}((tr_g A_g)^2-|A_g|^2)$
we note that $\sigma_1(g) =R_g/(2(n-1))$
and
\begin{equation}\label{eq034}
    \sigma_2(g) = \displaystyle \frac{1}{2(n-2)^2}\left(\frac{n}{4(n-1)}R_g^2 -|Ric_g|^2\right)=\frac{1}{2(n-2)^2}\left(\frac{(n-2)^2}{4n(n-1)}R_g^2-|\mathring{Ric}_g|^2\right),
\end{equation}
where $\mathring{Ric}_g$ is the trace free Ricci tensor. We consider the $\sigma_2$-curvature as a map $\sigma_2:\mathcal M\rightarrow C^\infty(M)$, where $\mathcal M$ is the space of all Riemannian metrics on $M$ and $C^\infty(M)$ is the space of all smooth functions on $M$. It was proved in \cite{MR3828913} that the linearization of the $\sigma_2$-curvature map is the map $\Lambda_g:S_2(M)\rightarrow C^\infty(M)$ given by
 \begin{equation}\label{eq001}
 \begin{array}{rcl}
 c(n)\Lambda_g(h)  =  \displaystyle\left\langle Ric_g,-\Delta_E^gh+\nabla^2tr_gh+2\delta^*\delta h\right\rangle
\displaystyle -\frac{n}{2(n-1)}R_g\left(\Delta_g tr_gh-\delta^2h+\langle Ric,h\rangle\right),
 \end{array}
 \end{equation}
 where $c(n)=2(n-2)^2$, and the $L^2$-formal adjoint of $\Lambda_g$ is the map $\Lambda_g^*:C^\infty(M) \rightarrow S_2(M)$ given by
 \begin{equation}\label{eq031}
     \begin{array}{rcl}
c(n)\Lambda_g^*(f)  & = &  \displaystyle-\Delta_E^g(fRic_g)+\delta^2(fRic_g)g+2\delta^*\delta(fRic_g)
\\
& & -\displaystyle\frac{n}{2(n-1)}\left(\Delta_g(fR_g)g-\nabla^2(fR_g)+fR_gRic_g\right).
\end{array}
 \end{equation}
 Here $S_2(M)$ is the space of symmetric 2-tensors on $M$, $\delta=-\mbox{div}$, $\delta^*$ is the $L^2$-formal adjoint of $\delta$ which is given by $(\delta^*\alpha)_{ij}=\frac{1}{2}(\nabla_i\alpha_j+\nabla_j\alpha_i)$ for all $1$-tensors $\alpha$, $\mathring R(h)_{ij}=g^{kl}g^{st}R_{kijs}h_{lt}$ for all $h\in S_2(M)$, and $\Delta_E^g$ is the Einstein operator defined in \eqref{eq004}. This implies that
 \begin{equation}\label{eqn003}
  tr_g\Lambda_g^*(f)=-\mathcal T_g(f)\quad\mbox{ and }\quad \delta\Lambda_g^*(f)=\frac{1}{2}fd\sigma_2(g),
 \end{equation}
where $\mathcal T_g(u)$ is defined in \eqref{eq049}.  When $g$ is an Einstein metric, it holds
\begin{equation}\label{eq007}
    \sigma_2(g)=\frac{1}{8n(n-1)}R_g^2\qquad\mbox{ and }\qquad\mathcal T_g=\frac{1}{4n}R_g\Delta_g+\frac{1}{4n(n-1)}R_g^2.
\end{equation}

One recall the $L^2$-formal adjoint of the linearization of scalar curvature at a given metric $g$ which is given by  $L^*_g(f)=\nabla^2f-g\Delta_gf-fRic_g$ (see for instance \cite{MR0380907}). We observe that we can recover the Ricci tensor and the scalar curvature from $L^*_g$,  since $L^*_g(1)=-Ric_g$ and  $tr_g L^*_g(1)=-R_g.$ In a complete analogy we can define a symmetric two-tensor by $\Lambda_g^*(1)$. Moreover, we can recover the $\sigma_2$-curvature taking the trace of $\Lambda_g^*(1)$. In fact, this is a consequence of \eqref{eq049} and \eqref{eqn003}, or of the following lemma, which the proof is a direct computation.
\begin{lemma}\label{lem007}
For any metric $g$ we have
\begin{align*}
    2(n-2)^2\Lambda_g^*(1)  
     = &  \displaystyle -\Delta_E^g\mathring{Ric}_g+\frac{n-2}{2n(n-1)}(\Delta_gR_g)g-\frac{n-2}{2(n-1)}\nabla^2R_g\\
    & -\frac{(n-2)^2}{2n(n-1)}R_g\mathring{Ric}_g -\frac{2}{n}|\mathring{Ric}_g|^2g-\frac{4(n-2)^2}{n}\sigma_2(g)g.
\end{align*}
\end{lemma}

Note that $\mathcal T_g$ is self-adjoint. Then for any $v\in C^\infty_0(M)$ (the space of all smooth functions with compact support in $M$), using \eqref{eqn003} it holds
$$\int_Mv\Lambda_g\left(ug\right)dv_g=-\int_Mu\mathcal T_g(v)dv_g=-\int_Mv\mathcal T_g(u)dv_g,$$
which implies the following result.

\begin{lemma}\label{Lem008}
For any $u\in C^\infty(M)$ we have
$
    \Lambda_g(ug)=-\mathcal T_g(u).
$\end{lemma}

Finally, given $u\in C^\infty(M)$, the $\sigma_2$-curvature of the metric $g_0$ and of the conformal metric $g=e^{2u}g_{0}$ are related by the equation (see \cite{MR3663325}),

\begin{equation}\label{eq047}
   -\mathcal T_{g_0}(u)+\left(2u+\frac{1}{2}\right)\sigma_2(g_0)-\frac{1}{2}\sigma_2(g)e^{4u}+\mathcal I_{g_0}(u)=0,
\end{equation}
where
\begin{equation}\label{eq015}
\begin{array}{lcl}
  \mathcal I_g(u) & = &\displaystyle \frac{1}{4}\left((\Delta_g u)^2-|\nabla_g^2u|^2+\langle \nabla_g u,\nabla_g|\nabla_g u|^2\rangle\right)+\dfrac{(n-3)}{4}|\nabla_g u|^2\Delta_g u
  \\
 & &+\displaystyle \frac{(n-1)(n-4)}{16}|\nabla_g u|^4-\dfrac{1}{2(n-2)} Ric_g(\nabla_g u, \nabla_g u)  -\displaystyle\frac{(n-4)}{8(n-2)}R_g|\nabla_g u|^2.
\end{array}
\end{equation}

\section{Critical Metrics in Closed Manifolds}\label{closed section}

 In this section we consider a closed Riemannian manifold $(M^n,g)$ of dimension $n\geq 3$ satisfying the equation 
 \begin{equation}\label{eq028}
 \Lambda^*_g(f)=\kappa g.
\end{equation} for some constant $\kappa$ and some smooth function $f$. An interesting question is if it is possible to classify the Riemannian manifolds $(M,g)$ which satisfy \eqref{eq028} for some smooth function $f$. This question does not seem to be easy to deal with due to the complexity of the equation.

\subsection{Proof of Theorem \ref{theorem-A}}
We start this subsection with the following preliminary result.

\begin{theorem}\label{teo002}
Let $(M^{n},g)$ be a connected closed Riemannian manifold of dimension $n\geq 3$. Suppose there exists a non trivial smooth function $f\in C^\infty(M)$ satisfying \eqref{eq028}.
\begin{enumerate}
 \item[(a)]  If $g$ is an admissible metric, then the $\sigma_2$-curvature is constant.
 \item[(b)] If $f$ is a nonzero constant $c$, then the $\sigma_2$-curvature is constant equal to $-\frac{n\kappa}{2c}$.
\end{enumerate}
\end{theorem}

\begin{proof}
By \eqref{eqn003} and \eqref{eq028} we obtain
$0=\delta\Lambda_g^*(f)=\frac{1}{2}fd\sigma_2(g).$
Suppose there exists $p\in M$ with $f(p)=0$ and $d\sigma_2(g)\neq 0$ at $p$. By taking derivatives, we can see that  $\nabla^mf(p)=0$ for all $\ m\geq 1.$ Moreover, note that by \eqref{eqn003} the function $f$ satisfies 
\begin{equation}\label{eq021}
\mathcal T_g(f)=-n\kappa,
\end{equation}
where $\mathcal T_g$ is defined in \eqref{eq049}. Since $g$ is an admissible metric, then (\ref{eq021}) is an elliptic equation. By results in \cite{MR76155} and \cite{MR0086237} we can conclude that $f$ vanishes identically in $M.$ But this is a contradiction. Therefore, $d\sigma_2(g)$ vanishes in $M$ and thus $\sigma_2(g)$ is constant.
 
Now, if $f$ is a nonzero constant $c$, by (\ref{eqn003}) we obtain $n\kappa=tr_g\Lambda_g^*(c)=-2\sigma_2(g)c$.
\end{proof}

When in \eqref{eq028} the function $f$ is constant we obtain the following result, which is a consequence of the results in \cite{MR2052939,MR1872547}.

\begin{proof}[Proof of Theorem \ref{theorem-A}]
Direct computations show the sufficiency. Therefore, we need only to prove the ``only if" part. Then, suppose that $g$ is a $\sigma_2$-Einstein metric with $\sigma_2(g)\geq 0$. By Theorem \ref{teo002}  $\sigma_2(g)$ is constant and we can write
$\Lambda_g^*(1)=-\frac{2}{n}\sigma_2(g)g.$

Now, define the Riemannian functional $\hat{\mathcal F_2}:\mathcal M\rightarrow\mathbb R$, where $\mathcal M$ is the space of all Riemannian metrics in $M$, given by
\begin{equation}\label{eq051}
    \hat{\mathcal F_2}(g)=V(g)^{\frac{4-n}{n}}\int_M\sigma_2(g)dv_g,
\end{equation}
where $V(g)$ is the volume of $(M,g).$ For a general $k$ the functional $\hat{\mathcal F_k}$ associated to the $\sigma_k$-curvature was introduced and discussed in \cite{MR1738176}.

Note that $\hat{\mathcal F_2}(\lambda g)=\hat{\mathcal F_2}(g)$ for all positive constant $\lambda$. Also, the first variation of $\hat{\mathcal F_2}$ is given by
\begin{align*}
    D\hat{\mathcal F_2}(g)(h)  = &~\displaystyle V(g)^{\frac{4-n}{n}}\displaystyle\int_M\left(\Lambda_g(h)+\frac{1}{2}\sigma_2(g)tr_gh\right)dv_g
 \displaystyle+\frac{4-n}{2n}V(g)^{\frac{4-2n}{n}}\int_M\sigma_2(g)dv_g\int_Mtr_ghdv_g
\\
 = & ~ \displaystyle V(g)^{\frac{4-n}{n}}\int_M\left\langle \Lambda_g^*(1)+\frac{2}{n}\sigma_2(g)g,h \right\rangle dv_g,
\end{align*}
where in the second equality we used that $\Lambda^*_g$ is the $L^2$-formal adjoint of $\Lambda_g$ and that $\sigma_2(g)$ is constant. This implies that a $\sigma_2$-Einstein metric is a critical point for the functional $\hat{\mathcal F_2}$. Therefore, the result follows by \cite[Theorem 1.1]{MR1872547} and \cite[Theorem B]{MR2052939}.
\end{proof}

We remark that  the case $n=4$ in Theorem \ref{theorem-A} cannot be treated on account of the specificity of this dimension.
Indeed, for a locally conformally flat manifold the Weyl tensor vanishes. Taking derivatives of the Gauss-Bonnet-Chern Formula \eqref{Gauss-Bonnet-int}, for all symmetric 2-tensor $h$, we obtain
$$0=\int_M\left(\Lambda_g(h)+\frac{1}{2}\sigma_2(g)tr_gh\right)dv_g=\int_M\left\langle \Lambda_g^*(1)+\frac{1}{2}\sigma_2(g)g,h\right\rangle dv_g.$$
This implies that $\Lambda_g^*(1)=-\frac{1}{2}\sigma_2(g)g$. Therefore, all Riemannian manifold of dimension 4 is $\sigma_2$-Einstein. This is a counterpart to Einstein manifold where all Riemannian manifold of dimension 2 is Einstein.

The proof of Theorem \ref{theorem-A} tells us that a $\sigma_2$-Einstein metric is a critical point for the functional $\hat{\mathcal F_2}$ (see \eqref{eq051}). From this and the example given in \cite[Section 6]{MR2052939} and \cite[Remark 7.1]{MR1872547}, we obtain the existence of $\sigma_2$-Einstein metrics with $\sigma_2(g)<0$ which are not Einstein.

Restricting the functional $\mathcal F_2(g)=\int_M\sigma_2(g)dv_g$ to a certain space of metrics, it is possible to give conditions for its  critical metrics to be hyperbolic, see Theorem 1.2 of \cite{MR1872547}, which holds as long as  Theorem B of \cite{MR2052939} holds.  
Arguing similarly as in Theorem \ref{theorem-A} we have

\begin{theorem}
Let $(M^{n},g)$ be a closed Riemannian manifold of dimension $n\geq 3$, $n\not=4$, which is locally conformally flat for $n\geq 5$. If $(M,g)$ is a $\sigma_2$-Einstein manifold with $\sigma_2(g)>0$ and scalar curvature $R_g<0$, 
then $(M^n,g)$ is hyperbolic.
\end{theorem}

\subsection{Second Variation of \texorpdfstring{$\mathcal F_2$}{Lg}}
   
In this subsection, we will compute the second variation of the functional $\mathcal F_2$ at critical points in $\mathcal M$ . First, we give a formula for the second derivative of the $\sigma_2$-curvature. We remark that our notation convention for the curvature tensor gives
\begin{equation}\label{eq054}
    R_{ijkl} =\kappa (g_{il}g_{jk}-g_{ik}g_{jl})
\end{equation}
in the case that $g$ has constant sectional curvature $\kappa$. First, we remember the second derivative of the scalar curvature, which can be found in \cite[Lemma 3 ]{MR2546025}.
\begin{lemma}\label{lem005}
Let $\{g(t)\}$ be a smooth path of smooth metrics with $g(0)=g$. Let $R(t)$ be the scalar curvature of $g(t)$. Then
\begin{align*}
     R''(0)  = &~ \displaystyle\Delta_g|h|^2 +2\langle h,\nabla ^2tr_gh\rangle+4\langle \nabla \delta h,h\rangle-2\left|\delta h+\frac{1}{2}dtr_gh\right|^2-\frac{1}{2}|\nabla h|^2
     \\
       &\displaystyle+2\langle \mathring R(h),h\rangle -g^{pq}g^{ij}g^{st}\nabla_ph_{is}\nabla_th_{jq}+DR_g(h').
\end{align*}
where $h=g'(0)$ and $h'=g''(0)$, the covariant derivative, the curvature tensor are with respect to $g$, and $DR_g$ is the linearization of the scalar curvature at $g$.
\end{lemma}

Using this lemma we find the second derivative of the $\sigma_2$-curvature.

\begin{proposition}\label{lem004}
Let $\{g(t)\}$ be a smooth path of metrics with $g(0)=g$, $h=g'(0)$ and $h'=g''(0)$. Let $\sigma_2(t)$ be the $\sigma_2$-curvature of $g(t)$. Then
\begin{align*}
c_n\sigma_2''(0) 
 = & ~ \displaystyle-\frac{1}{2}\left| \Delta_E^g h-\nabla^2tr_gh-2\delta^*\delta(h)\right|^2+\left\langle Ric_g\circ h,2\Delta_E^g h- \nabla^2tr_gh-2\delta^*\delta h-Ric_g\circ h\right\rangle
 \\
& 
     -\left\langle Ric_g,\langle h,\nabla^2h_{ij}\rangle-4g^{lp}g^{mq}h_{lm}\nabla_p\nabla_jh_{iq}+2g^{pq}g^{sl}\nabla_ph_{sj}(\nabla_qh_{il}-\nabla_lh_{qi})+\nabla^2|h|^2\right.
     \\
     &  \displaystyle-g^{sl}h_{sj}\left(\nabla_i(\delta h)_l-\nabla_l(\delta h)_i\right)+\frac{1}{2} g^{kl}\left(\nabla_jh_{il}-3\nabla_lh_{ij}\right)\left((\delta h)_k+\frac{1}{2}\nabla_k(tr_gh)\right)
      \\
       &   \displaystyle \left.-2 \delta\left(h\circ \left(\delta h+\frac{1}{2}\nabla tr_gh\right)\right)+\langle h,\nabla^2h_{ij}\rangle -g^{pq}g^{sl}\nabla_jh_{sq}\nabla_ih_{pl}-g^{sl}g^{mt}h_{sj}h_{mi}R_{lijt}\right\rangle\\
  & \displaystyle +\frac{n}{2(n-1)}\left(\left(\Delta tr_gh-\delta^2h+\langle Ric,h\rangle\right)^2+R_g\left(\Delta_g|h|^2_g +2\langle h,\nabla^2tr_gh\rangle-\frac{1}{2}|\nabla h|^2\right.\right.\\
 &   \displaystyle \left.\left.+4\langle \nabla\delta h,h\rangle-2\left|\delta h+\frac{1}{2}\nabla tr_gh\right|^2\right.\left. -g^{pq}g^{ij}g^{st}\nabla_ph_{is}\nabla_th_{jq}+2\langle \mathring R(h),h\rangle\right)\right) +c_n\Lambda_g(h'),  
\end{align*}
    where $c_n=2(n-2)^2$, all covariant derivative are with respect to $g$, $R_{lijt}$ is the curvature tensor of $g$, and $\Lambda_g$ is the linearization of the $\sigma_2$-curvature map given by \eqref{eq001}.
\end{proposition}
\begin{proof}
For any fixed point $p$, let $\{x_i\}$ be a normal coordinate chart at $p$ with respect to $g(0)=g$. In these coordinates the Christoffel symbols are equal to zero at $p$. It is well known the following variational formulae
\begin{align*}
   \frac{\partial}{\partial t}\Gamma_{ij}^k & =  \displaystyle\frac{1}{2}g^{kl}(\nabla_ih_{jl}+\nabla_jh_{il}-\nabla_lh_{ij}),\\ 
   \frac{\partial}{\partial t}R_{pij}^l 
& =       \displaystyle \frac{1}{2}g^{lm}\left(\nabla_p\nabla_jh_{im}-\nabla_p\nabla_mh_{ij}-\nabla_i\nabla_jh_{pm}+\nabla_m\nabla_ih_{pj}\right.\\
&\left.+g^{ar}R_{ipjr}h_{am}+g^{ar}R_{ipmr}h_{ja}  -g^{ar}R_{impr}h_{aj}-g^{ar}R_{imjr}h_{pa}\right),\\
   Ric_g' & =  \displaystyle -\frac{1}{2}\left(\Delta_gh+\nabla^2tr_gh+2\delta^*\delta(h)+2\mathring{R}(h)-Ric\circ h-h\circ Ric_g\right),\\
   R_g'  & =  \displaystyle  -\Delta_gtr_gh+\delta^2h-\langle h,Ric_g\rangle.
\end{align*}

From now on, all derivatives are taken at $t=0$. By \eqref{eq001} we have
\begin{eqnarray}
         c_n\sigma_2''(0) 
   = & \displaystyle\left\langle 2Ric_g\circ h-Ric_g',\Delta_E^gh-\nabla^2tr_gh-2\delta^*\delta h\right\rangle-\displaystyle\left\langle Ric_g,(\Delta_E^gh-\nabla^2tr_gh-2\delta^*\delta h)'\right\rangle\nonumber\\
 & \displaystyle +\frac{n}{2(n-1)}\left(\left(\Delta_g tr_gh-\delta^2h+\langle Ric,h\rangle\right)^2+R_gR_g''(0)\right). \label{eq053}
\end{eqnarray}

First, we have
$
     \mathring R(h)_{ij}'  =  \mathring R(h')_{ij} - 2 \mathring R(h\circ h)_{ij}+g^{pq}g^{st}R_{pijs}'h_{qt}
$
and $R_{pijs}'=g^{lm}h_{sl}R_{pijm}+g_{sl} \dfrac{\partial}{\partial t}R_{pij}^l $,
which implies that
$$\langle Ric_g,\mathring R(h)'\rangle =\displaystyle \left\langle Ric_g,\mathring R(h')-\mathring R(h\circ h)+g^{pq}g^{tm}h_{qt}\left(\nabla_p\nabla_jh_{im}- \frac{1}{2}(\nabla_p\nabla_mh_{ij}+\nabla_i\nabla_jh_{pm})\right)\right\rangle.$$

Using the variations formula for the geometric quantities above, the Ricci Identity and that the Ricci tensor is a symmetric tensor, we obtain
$$\langle Ric_g,(\nabla^2tr_gh)'\rangle  = \displaystyle \left\langle Ric_g,\nabla^2tr_gh'-\nabla^2|h|_g^2-g^{kl}\left(\nabla_jh_{il}-\frac{1}{2}\nabla_lh_{ij}\right)\nabla_k(tr_gh)\right\rangle,$$
\begin{align*}
    \left\langle Ric_g,(\Delta_gh_{ij})' \right\rangle  = &  \left\langle Ric_g,
     \Delta_gh_{ij}' -\langle h,\nabla^2h_{ij}\rangle-g^{pq}\left(2\nabla_p((\Gamma_{qi}^s)'h_{sj})+(\Gamma_{pq}^s)'\nabla_sh_{ij}+2(\Gamma_{pi}^s)'\nabla_qh_{sj}\right)\right\rangle\\
           = & 
     \left\langle Ric_g,\Delta_gh'-\langle h,\nabla^2h_{ij}\rangle-h\circ(\Delta_gh)\right. -\left.2g^{pq}g^{sl}\nabla_ph_{sj}(\nabla_qh_{il}+\nabla_ih_{ql}-\nabla_lh_{qi})\right.
      \\
      &  \displaystyle\left.+g^{sl}h_{sj}\left(\nabla_i(\delta h)_l-\nabla_l(\delta h)_i\right)\right. +\left.g^{sl}\left((\delta h)_l+\frac{1}{2}\nabla_ltr_gh\right)\nabla_sh_{ij}\right\rangle\\
       & 
      \displaystyle-\langle h\circ h,Ric_g\circ Ric_g\rangle+g^{ia}g^{jb}g^{sl}g^{mt}h_{sj}h_{mi}R_{ab}R_{lt},\\
\langle Ric_g,\delta^*\delta(h)'\rangle = & - \displaystyle g^{ia}g^{jb}R_{ab}(g^{pq}\nabla_j\nabla_ph_{iq})'=\langle Ric_g,\delta^*\delta(h')\rangle
\\
&  \displaystyle + g^{ia}g^{jb}g^{lp}g^{mq}R_{ab}h_{lm}\nabla_j\nabla_ph_{iq}+\left\langle Ric_g,g^{pq}(\nabla_j((\Gamma_{pi}^s)'h_{sq}+(\Gamma_{pq}^s)'h_{is})\right.
\\
&  \left. +(\Gamma_{jp}^s)'\nabla_sh_{iq}  +(\Gamma_{ji}^s)'\nabla_ph_{sq}+(\Gamma_{jq}^s)'\nabla_ph_{is})\right\rangle 
\\
 = &\displaystyle  \left\langle Ric_g,\delta^*\delta(h')+\mathring R(h\circ h)+ g^{lp}g^{mq}h_{lm}\nabla_p\nabla_jh_{iq}\right.- g^{pq}g^{sl}\nabla_jh_{pl}\nabla_sh_{iq}
\\
& \displaystyle +\frac{1}{2}g^{pq}g^{sl}\left(\nabla_jh_{sq}\nabla_ih_{pl}  +h_{sq}\nabla_j\nabla_ih_{pl} \right) -g^{sl}\nabla_j\left(h_{is}\left(\left(\delta h\right)_l+\frac{1}{2}\nabla_ltr_gh\right)\right)
\\
&  \displaystyle\left.+ g^{sl}\left(\nabla_ih_{jl}-\frac{1}{2}\nabla_lh_{ji}\right)(\delta h)_s\right\rangle- \langle \mathring R(h),Ric_g\circ h\rangle.   
\end{align*}

Therefore, the result follows by these evolution equations, Lemma \ref{lem005} and \eqref{eq053}.
\end{proof}

A direct consequence is the following.

\begin{coring}\label{cor003}
Suppose $(M^n,g)$ has constant sectional curvature $\kappa$.  Let $\{g(t)\}$ be a smooth path of metrics with $g(0)=g$, $h=g'(0)$ and $h'=g''(0)$. Let $\sigma_2(t)$ be the $\sigma_2$-curvature of $g(t)$. Then
 $$\sigma_2''(0) 
  =  \Lambda_g(h')+\frac{1}{2(n-2)^2}I,$$
 where
\begin{align*}
    I
  = &  \displaystyle-\frac{1}{2}\left| \Delta h+\nabla^2tr_gh+2\delta^*\delta(h)+2\kappa ((tr_gh)g-h)\right|^2+(n-2)^2\kappa^2((tr_gh)^2-|h|^2)
 \\
 &  +\dfrac{(n-2)^2}{2}\kappa\left(\Delta|h|^2-g^{ij}g^{pq}g^{sl}\nabla_ph_{sj}\nabla_lh_{qi}\right)+\dfrac{n}{2(n-1)}\left(\Delta tr_gh-\delta^2h+\kappa(n-1)tr_gh\right)^2
 \\
      &  
      +\kappa\left\langle h,(n^2-3n+3)\nabla^2tr_gh-2(n^2-n+1)\delta^*\delta h\right\rangle-(n^2+2n-2)\kappa \left|\delta h+\frac{1}{2}\nabla tr_gh\right|^2\\
       &  \displaystyle +2(n-1)\kappa\left( \delta \left(h\circ \left(\delta h+\frac{1}{2}\nabla tr_gh\right)\right)-\left\langle \delta h+\frac{1}{2}\nabla tr_gh, \frac{1}{2}\nabla tr_gh\right\rangle\right)-\frac{(n-2)^2}{4}\kappa |\nabla h|^2.
\end{align*}
    Here  all covariant derivatives are with respect to $g$ and $\Lambda_g$ is the linearization of the $\sigma_2$-curvature map given by \eqref{eq001}.
\end{coring}

As an application of Theorem \ref{theorem-A} we have the following simplification of the second variation formula of the functional total $\sigma_2$-curvature.

\begin{proposition}\label{sv}
Let $(M^{n},g)$ be a closed Riemannian manifold of dimension $n\geq 3$, $n\not=4$, which is locally conformally flat for $n\geq 5$. Suppose $g$ is $\sigma_2$-Einstein with $\sigma_2(g)=\frac{n(n-1)}{8}k^2$ and unit volume.  
Then the second derivative   
of $\mathcal F_2$ at $g$ restricted to the space of metrics with unit volume $\mathcal{M}_1$ in the direction $h\in T_g\mathcal{M}_1$ is given by
\begin{align*}
    \mathcal{F}_2^{\prime\prime}(h)
  = & ~ \displaystyle \frac{1}{2}\int_{M} \Lambda_g(h)tr_g h\;dv_g -\frac{1}{4(n-2)^2}  \int_{M} \left| \Delta h+\nabla^2tr_gh+2\delta^*\delta(h)+2\kappa ((tr_gh)g-h)\right|^2dv_{g}\\
  & 
  +\kappa^2  \int_{M} ((tr_gh)^2-|h|^2)dv_g+ \frac{\kappa}{4}\int_M\left(\Delta|h|^2-g^{ij}g^{pq}g^{sl}\nabla_ph_{sj}\nabla_lh_{qi}\right)dv_g
 \\
 & +\dfrac{n}{4(n-1)(n-2)^2}\int_M\left(\Delta tr_gh-\delta^2h+\kappa(n-1)tr_gh\right)^2dv_g
 \\
      &  
      +\frac{\kappa}{2(n-2)^2}\int_M\left\langle h,(n^2-3n+3)\nabla^2tr_gh-2(n^2-n+1)\delta^*\delta h\right\rangle dv_g\\
      & -\frac{(n^2+2n-2)\kappa}{2(n-2)^2}\int_M |\delta h+\frac{1}{2}\nabla tr_gh|^2dv_g-\frac{(n-1)}{4}\kappa^2\left(\int_M|h|^2dv_{g}-\frac{1}{2}\int_M (tr_{g} h)^{2}dv_g\right)\\
       &  \displaystyle -\frac{(n-1)}{(n-2)^2}\kappa\int_M \left\langle \delta h+\frac{1}{2}\nabla tr_gh, \frac{1}{2}\nabla tr_gh\right\rangle dv_g-\frac{\kappa}{8}\int_M |\nabla h|^2dv_g.
\end{align*}

\end{proposition}

\begin{proof}

Consider a one parameter deformation $g(t)$
of $g$ in $\mathcal M_1$ with $g'(0)=h$.  
Since $(M,g)$ is a $\sigma_2$-Einstein manifold, we have that $\sigma_2(g)$ is constant (Theorem \ref{teo002}). Thus  
\begin{eqnarray*}
\left.\frac{d^{2}}{d t^{2}}\right|_{t=0} \int_M\sigma_2(g(t))dv_{g_t}
&=&\int_M\sigma_2''(g)dv_{g}+\frac{1}{2}\int_{M} \Lambda_g(h)tr_g h\;dv_g,
\end{eqnarray*}
where we have used that the volume of $g(t)$ is unit. Moreover, for $n\geq 3$, $n\not=4$, and assuming  local conformal flatness for $n\geq 5,$  a  $\sigma_2$-Einstein metric is Einstein by Theorem \ref{theorem-A}. Since
$$0=\left.\frac{d^{2}}{d t^{2}}\right|_{t=0} \operatorname{V}\left(M, g(t)\right) =\frac{1}{2} \int_{M}\left[\operatorname{tr}_{g} h^\prime+(1 / 2)\left(\operatorname{tr}_{g} h\right)^{2}-|h|_{g}^{2}\right] d v_{g}, $$ 
where $h'=g''(0)$, then the result follows from Corollary \ref{cor003} and the fact that $(M,g)$
is a $\sigma_2$-Einstein manifold, which implies
$$\int_M\Lambda_g(h')dv_g =  \displaystyle \int_M\langle h',\Lambda_g^*(1)\rangle dv_g=-\frac{(n-1)}{4}\kappa^2\int_M tr_gh'dv_g.$$
\end{proof}

Considering a manifold with constant sectional curvature,
we show that there always exists an infinite-dimensional subspace of tensors in $S_{2,g}^{TT}(M)$ on which the second variation of $\mathcal F_2$ at $t = 0$ is negative definite. Remember from introduction that 
$S_{2,g}^{TT}(M):=\{h \in S_2(M): \delta_g h = 0, tr_gh  = 0\}.$

\begin{coring}
Let the setting be as in  Proposition \ref{sv}. If the  sectional curvature $\kappa$ is a positive  constant, then restricted to variations $g(t)$ with $h:=g'(0)\in S_{2,g}^{TT}(M)$,  then $\mathcal F_2''(h)<0$.
\end{coring}
\begin{proof}
By Proposition \ref{sv}, for  $h\in S_{2,g}^{TT}(M)$ we have 
\begin{align*}
    \mathcal{F}_2^{\prime\prime}(h)=&-\frac{n+1}{4}\kappa^2\int_M|h|^2dv_{g}dv_g \displaystyle-\frac{1}{4(n-2)^2}\int_M\left| \Delta_gh-2\kappa h\right|^2dv_g\\
  &\displaystyle-\frac{\kappa}{8}\int_M |\nabla_gh|^2dv_g-\dfrac{\kappa}{4} \int_Mg^{ij}g^{pq}g^{sl}\nabla_ph_{sj}\nabla_lh_{qi}dv_g.
\end{align*}

 On the other hand,
 \begin{align*}
      \displaystyle \int_Mg^{ij}g^{pq}g^{sl}&\nabla_ph_{sj}\nabla_lh_{qi}  = -\displaystyle \int_Mg^{ij}g^{pq}g^{sl}h_{sj}\nabla_p\nabla_lh_{qi}\\
      & = \displaystyle \int_Mg^{ij}g^{pq}g^{sl}( -h_{sj}\nabla_l\nabla_ph_{qi}+g^{mt}h_{sj}R_{plqt}h_{mi}+g^{mt}h_{sj}R_{plit}h_{qm})\\
      & =  \displaystyle -n\kappa\int_M|h|^2 +\int_M(\kappa (tr_gh)^2-2\langle h,\nabla div_gh\rangle+|div_gh|^2).
 \end{align*}
 Hence the result follows.
\end{proof}

We remark that if $\kappa=0$, then the second variation is strictly negative except for parallel $h$.

Since in a compact and hyperbolic $n$-manifold the smallest eigenvalue
of the rough Laplacian on $S_{2,g}^{TT}(M)$  is at least $n,$ we have the following result.

\begin{coring}
Let the setting be as in  Proposition \ref{sv}. 
 If $(M ,g)$ is a hyperbolic manifold, then it is locally strictly maximizing for $ \mathcal{F}_2$ with respect to variations  in $S_{2,g}^{TT}(M)$.
\end{coring}

\section{Volume comparison}\label{volume comparison}

The goal of this section is to prove the Theorem \ref{theorem-B}, which was motivated by results due to W. Yuan \cite{yuan2016volume} and Y. J. Lin and W. Yuan \cite{lin2021deformations} for the scalar curvature and $Q$-curvature context, respectively. Roughly speaking, this theorem says that we cannot increase the volume of an Einstein manifold increasing the $\sigma_2$-curvature. First we prove some variational formulae, which will be necessary to obtain the volume comparison.

\subsection{Some Variational Formulae}
Let $(M^n,g_0)$ be a closed manifold and consider the functional $\mathcal E_{g_0}:\mathcal M\rightarrow\mathbb R$ given by
$$\mathcal E_{g_0}(g)=V(g)^{\frac{4}{n}}\int_M\sigma_2(g)dv_{g_0},$$
where the volume form $dv_{g_0}$ does not depend on $g$. Remember from the introduction that $\mathcal M$ is the space of all Riemannian metrics on $M$. This implies that $\mathcal E_{g_0}(\lambda g)=\mathcal E_{g_0}(g)$ for all real number $\lambda>0$, since the volume form does not depend on $g$. Define the 2-symmetric tensor $B_{g_0}$ as

\begin{equation}\label{eq059}
    B_{g_0}:=-\frac{1}{2}\Lambda_{g_0}^*(1).
\end{equation}
By \eqref{eqn003} and \eqref{eq049}, this tensor satisfies $tr_{g_0}B_{g_0}=\sigma_2(g_0)$ and $div_{g_0}B_{g_0}=\frac{1}{4}\sigma_2(g_0)$. Its trace free part will be denoted by $\mathring B_{g_0}:=B_{g_0}-\frac{1}{n}\sigma_2(g_0)g_0.$ To simplify the notation, we will use the convention that $'$ and $''$ stand for first and second variations with respect to a certain $h\in S_2(M)$, respectively. For a $\sigma_2$-Einstein metric $g_0$, a direct computation gives us
\begin{align*}
\mathcal E_{g_0}'(g_0)  = \displaystyle  -2V(g_0)^{\frac{4}{n}}\int_M\left\langle \mathring B_{g_0},h\right\rangle dv_{g_0}.
\end{align*}
This implies the following lemma.
\begin{lemma}\label{lemcritical}
A $\sigma_2$-Einstein manifold $(M,g_0)$ is a critical point to the functional $\mathcal E_{g_0}$.
\end{lemma}

Before we find the second variation of $\mathcal E_{g_0}$ let us prove the next lemma. 
\begin{lemma}\label{lem006}
For any $\sigma_2$-Einstein  metric $g$ we have
$$\int_M\sigma_2''(g)dv_g= -2\int_M\left(\langle \mathring B_g',\mathring h\rangle-\frac{1}{n}\sigma_2(g)|\mathring h|^2+\left(\frac{n+4}{4n}\Lambda_g(h)+\frac{n-2}{2n^2}\sigma_2(g)(tr_gh)\right)(tr_gh)\right)dv_g.$$
\end{lemma}
\begin{proof}
For any metric $g$ we have
\begin{align*}
    \int_M&\sigma_2''(g)  dv_g  = \left(\int_M\sigma_2(g)dv_g\right)''-2\int_M\sigma_2'(g)(dv_g)'-\int_M\sigma_2(g)(dv_g)''= -2\left(\int_M\langle B_g,h\rangle dv_g\right)'\\
    & -\int_M\sigma_2'(g)(dv_g)' = -2\int_M\left(\langle B_g',h\rangle-2\langle B_g,h\circ h\rangle+\frac{1}{2}\langle B_g,h\rangle(tr_gh)+\frac{1}{4}\Lambda_g(h)(tr_gh)\right) dv_g.
\end{align*}
Using that $\mathring B_{g}=0$, for a $\sigma_2$-Einstein metric $g$, and $\mathring B_g'=B_g'-\frac{1}{n}\Lambda_g(h)g-\frac{1}{n}\sigma_2(g)h$, taking the derivative of both sides in $tr_gB_g=\sigma_2(g)$, we find $tr_g\mathring B_g' =0$. Thus, considering that $h=\mathring h+\frac{1}{n}(tr_gh)g$, with $tr_g\mathring h=0$, we  conclude our lemma.

\end{proof}

\begin{proposition}\label{proposition002}
The second variation of $\mathcal E_{g_0}$ at a $\sigma_2$-Einstein metric $g_0$ is given by
\begin{align*}
   V(g_0)^{-\frac{4}{n}}   \mathcal E_{g_0}''(g_0)  = -2\int_M\left\langle (D\mathring B_{g_0})\left(\mathring h\right),\mathring h\right\rangle dv_{g_0}-\frac{n+4}{2n}\int_M\Lambda_{g_0}(\mathring h)(tr_{g_0}h)dv_{g_0}\\
  -\frac{2}{n}\int_M\left(tr_{g_0}\left((D\mathring B_{g_0})^*(\mathring h)\right)\right) (tr_{g_0}h)dv_{g_0}+\frac{n+4}{2n^2}\int_M(tr_{g_0}h-\overline{tr_{g_0}h})\mathcal T\left(tr_{g_0}h-\overline{tr_{g_0}h}\right)dv_{g_0}.
\end{align*}
\end{proposition}
\begin{proof}
Since $B_{g_0}=\frac{1}{n}\sigma_2(g_0)g_0$, we obtain
\begin{align*}
\mathcal E_{g_0}''&(g_0)
  =~\displaystyle V(g_0)^{\frac{4}{n}}\int_M\sigma_2''(g)dv_{g_0}-\frac{8}{n^2}\sigma_2(g_0)V(g_0)^{\frac{4}{n}-1}\left(\int_Mtr_{g_0}hdv_{g_0}\right)^2\\
 & +\frac{1}{n}V(g_0)^{\frac{4}{n}}\sigma_2(g_0)\left(\frac{4-n}{n}V(g_0)^{-1}\left(\int_Mtr_{g_0}hdv_{g_0}\right)^2+\int_M\left(\frac{n-2}{n}(tr_{g_0}h)^2-2|\mathring h|^2_{g_0}\right)dv_{g_0}\right).
\end{align*}

By Lemma \ref{lem006}, we have

\begin{align*}
    V&(g_0)^{-\frac{4}{n}}   \mathcal E_{g_0}''(g_0)\\  
  &=  \displaystyle -2\int_M\left(\langle \mathring B_g',\mathring h\rangle-\frac{1}{n}\sigma_2(g_0)|\mathring h|^2+\left(\frac{n+4}{4n}\Lambda_{g_0}(h)+\frac{n-2}{2n^2}\sigma_2(g_0)(tr_{g_0}h)\right)(tr_{g_0}h)\right)dv_{g_0}
\end{align*}

\begin{align*}
 & -\frac{n+4}{n^2}\sigma_2(g_0)V(g_0)^{-1}\left(\int_Mtr_{g_0}hdv_{g_0}\right)^2+\frac{1}{n}\sigma_2(g_0)\int_M\left(\frac{n-2}{n}(tr_{g_0}h)^2-2|\mathring h|^2_{g_0}\right)dv_{g_0}\\
 & = -2\int_M\langle \mathring B_g',\mathring h\rangle dv_{g_0}-\frac{n+4}{2n}\int_M\Lambda_{g_0}(\mathring h)(tr_{g_0}h)dv_{g_0}-\frac{n+4}{n^2}\sigma_2(g_0)V(g_0)^{-1}\left(\int_Mtr_{g_0}hdv_{g_0}\right)^2\\
 & -\frac{n+4}{2n^2}\int_M(tr_{g_0}h)\left(tr_{g_0}(\Lambda_{g_0}^*(tr_{g_0}h))\right)dv_{g_0}.
\end{align*}
Using that $\mathring B_{g_0}'=D\mathring B_{g_0}(h)$ and $h=\mathring h+\frac{1}{n}(tr_{g_0}h)g_0$, with $tr_{g_0}\mathring h=0$, we find
\begin{align*}
    -2\int_M\langle \mathring B_g',\mathring h\rangle dv_{g_0} 
    & = -2\int_M\left\langle (D\mathring B_{g_0})\left(\mathring h\right),\mathring h\right\rangle dv_{g_0} -\frac{2}{n}\int_M\left(tr_{g_0}\left((D\mathring B_{g_0})^*(\mathring h)\right)\right) (tr_{g_0}h)dv_{g_0}.
\end{align*}

Also, by \eqref{eq049} and \eqref{eqn003}, we obtain
\begin{align*}
   & -\frac{1}{2}\int_M(tr_{g_0}h)\left(tr_{g_0}(\Lambda_{g_0}^*(tr_{g_0}h))\right)dv_{g_0}-\sigma_2(g_0)V(g_0)^{-1}\left(\int_Mtr_{g_0}hdv_{g_0}\right)^2\\
  &=  \frac{1}{2}\int_M(tr_{g_0}h)\mathcal T_{g_0}(tr_{g_0}h)dv_{g_0}-\sigma_2(g_0)V(g_0)(\overline{tr_{g_0}h})^2= ~\frac{1}{2}\int_M(tr_{g_0}h-\overline{tr_{g_0}h})\mathcal T\left(tr_{g_0}h-\overline{tr_{g_0}h}\right)dv_{g_0},
\end{align*}
where $\overline{tr_{g_0}h}=V(g_0)^{-1}\displaystyle\int_Mtr_{g_0}hdv_{g_0}.$ 
From this and the previous formulas  we obtain the result.
\end{proof}
\begin{proposition}\label{propo003}
Let $g$ be an Einstein metric. Then for any $\mathring h\in S_{2,g}^{TT}(M)$, it holds
\begin{align*}
    D\mathring B_g(\mathring h) 
    = & \frac{1}{4(n-2)^2}\left(\Delta_E^g+2(n-2)^2\sigma_2(g)\right)(\Delta_E^g\mathring h).
\end{align*}
\end{proposition}
\begin{proof}
Since $g$ is Einstein, then $\mathring{Ric}_g\equiv 0$ and $R_g$ is constant. Also, $\mathring h\in S_{2,g}^{TT}(M)$ implies that 
$R_g'=-\Delta_g(tr_g\mathring g)+\delta^2\mathring h-\langle Ric_g,\mathring h\rangle=0$
and $\sigma_2'(g)=0$ (see \eqref{eq034}). By Lemma \ref{lem007} we obtain
\begin{align*}
    4(n-2)^2\mathring B_g =  
     \Delta_E^g\mathring{Ric}_g-\frac{n-2}{2n(n-1)}(\Delta_gR_g)g+\frac{n-2}{2(n-1)}\nabla^2R_g+\frac{(n-2)^2}{2n(n-1)}R_g\mathring{Ric}_g +\frac{4}{n}|\mathring{Ric}_g|^2g.
\end{align*}
Thus the result follows by taking derivatives and using Lemma 3.2 in \cite{lin2021deformations}, which gives us 
$D\mathring{Ric}_g(\mathring h)=\frac{1}{2}\Delta_E^g\mathring h.$
\end{proof}

We notice that the definition \eqref{eq004} of the Einstein operator differs from \cite[Definition 1.6]{lin2021deformations} and \cite[Definition 1.6]{lin2016deformations} by a sign.
Using \eqref{eq007} we have the following.

\begin{coring}\label{corollary001}
Let $g$ be an Einstein metric. Then $D\mathring B_g$ is a self-adjoint operator on $S_{2,g}^{TT}(M)$. In addition, if $g$ is a strictly stable Einstein metric, then $D\mathring B_g$ is positive.
\end{coring}
\begin{proposition}\label{proposition001}
Let $g$ be a  Riemannian metric. For
$h\in S_{2,g}^{TT}(M)\oplus(C^\infty(M)\cdot g)$ we have
$$\Lambda_g(h)= \displaystyle\frac{1}{2(n-2)^2}div\left(\langle \mathring{Ric}_g,\Delta_g h \rangle-\langle h,\Delta_g \mathring{Ric}_g \rangle+\frac{n-2}{2(n-1)}\mathring h(\nabla R_g,\cdot)\right)-2\left\langle \mathring B_g,\mathring h\right\rangle-\frac{1}{n}\mathcal T(tr_gh),$$
where $\mathring h$ is the trace free part of $h$.
\end{proposition}
\begin{proof}
By Lemma \ref{Lem008} we get
$\Lambda_g\left(\frac{1}{n}(tr_gh)g\right)=-\frac{1}{n}\mathcal T(tr_gh).$
Using \eqref{eq001} and Lemma \ref{lem007} we obtain

$$\Lambda_g(\mathring h) +2\left\langle \mathring B_g,\mathring h\right\rangle=  \displaystyle-\frac{1}{2(n-2)^2}\left\langle \mathring{Ric}_g,\Delta_E^g\mathring{h}+\frac{(n-2)^2}{2n(n-1)}R_g\mathring h\right\rangle$$

\begin{align*}
    & +\frac{1}{2(n-2)^2}\left\langle \Delta_E^g\mathring{Ric}_g+\frac{n-2}{2(n-1)}\nabla^2R_g+\frac{(n-2)^2}{2n(n-1)}R_g\mathring{Ric}_g,\mathring h\right\rangle\\
    = &~ \displaystyle\frac{1}{2(n-2)^2}div\left(\langle \mathring{Ric}_g,\Delta_g h \rangle-\langle h,\Delta_g \mathring{Ric}_g \rangle+\frac{n-2}{2(n-1)}\mathring h(\nabla R_g,\cdot)\right).
\end{align*}
\end{proof}

\begin{coring}\label{cor004}
Let $(M,g_0)$ be an Einstein manifold. For $h\in S_{2,g_0}^{TT}(M)\oplus (C^\infty(M)\cdot g_0)$ we have
\begin{align*}
    V(g_0)^{-\frac{4}{n}}  \mathcal E_{g_0}''(g_0)=    \frac{n+4}{2n^2}\int_M(tr_{g_0}h-\overline{tr_{g_0}h})\mathcal T_{g_0}\left(tr_{g_0}h-\overline{tr_{g_0}h}\right)dv_{g_0}-2\int_M\left\langle (D\mathring B_{g_0})\left(\mathring h\right),\mathring h\right\rangle dv_{g_0},
\end{align*}
where $\mathring h$ is the trace free part of $h$.
\end{coring}
\begin{proof} By Corollary \ref{corollary001}, $D\mathring B_{g_0}$ is a self-adjoint operator on $S_{2,g_0}^{TT}(M)$. Since $\mathring B_{g_0}=0$ when $g_0$ is Einstein, and $tr_{g}\mathring B_{g}=0$, for any metric $g$, then differentiating $tr_{g}\mathring B_{g}=0$ with respect to $g$ we get
$tr_{g_0}\left(D\mathring B_{g_0}(h)\right)=\langle \mathring B_{g_0},h \rangle.$
This implies that
$tr_{g_0}\left((D\mathring B_{g_0})^*(\mathring h)\right)=tr_{g_0}\left(D\mathring B_{g_0}(\mathring h)\right)=\langle \mathring B_{g_0},\mathring h \rangle=0.$
Also, by Proposition \ref{proposition001}, we obtain that $\Lambda_{g_0}(\mathring h)=0$. The result follows by Proposition \ref{proposition002}.
\end{proof}

\begin{proposition}\label{propo004}
Let $(M,g_0)$ be  a strictly stable Einstein manifold with $Ric_{g_0}=(n-1)\lambda g_0$, where $\lambda\geq 0$ is a constant. Then $g_0$ is a critical point of $\mathcal E_{g_0}$ and $D^2\mathcal E_{g_0}(h,h)\leq 0$ for any $h\in S_{2,g_0}^{TT}(M)\oplus(C^\infty(M)\cdot g_0)$. Moreover, the equality holds if and only if 
\begin{enumerate}
    \item[(a)] $h\in C^\infty(M)\cdot g_0$, for $\lambda =0$.
    \item[(b)] $h\in(\mathbb R\oplus E_{\lambda})g_0$, for $\lambda>0$ and $(M,g_0)$ is isometric to the round sphere with radius $\frac{1}{\sqrt{\lambda}}$.
    \item[(c)] $h\in\mathbb Rg_0$, for $\lambda>0$ and $(M,g_0)$ is not isometric to the round sphere, up to rescaling.
\end{enumerate}
Here $E_\lambda:=\{u\in C^\infty(\mathbb S^n(\frac{1}{\sqrt{\lambda}})):\Delta_{\mathbb S^n(\frac{1}{\sqrt{\lambda}})} u+n\lambda u=0\}$.
\end{proposition}
\begin{proof}
The metric $g_0$ is a critical point of $\mathcal E_{g_0}$ by Lemma \ref{lemcritical}, since an Einstein metric is a $\sigma_2$-Einstein metric. By \eqref{eq007} we have $\sigma_2(g)\geq 0$ and $\mathcal T_{g_0}=\frac{n-1}{4}\lambda(\Delta_{g_0}+n\lambda)$. By the Lichnerowicz-Obata's Theorem (see \cite[Chapter 3]{MR1333601}, for instance) we get that the first eigenvalue of $\mathcal T_{g_0}$ is greater than equal to $n\lambda$. By Proposition \ref{propo003} and Corollary \ref{cor004} we obtain the first part of the result.

Since $g_0$ is stricly stable, by Corollary \ref{cor004} the equality $D^2\mathcal E_{g_0}(h,h)= 0$ implies that $h=fg$, for some $f\in C^\infty(M)$, and
$\lambda\int_M(f-\overline f)(\Delta_{g_0}+n\lambda)(f-\overline f)dv_{g_0}=0.$
This and the Lichnerowicz-Obata's Theorem implies the result.
\end{proof}

\subsection{Volume comparison  with respect to \texorpdfstring{$\sigma_2$}{Lg}-curvature}

Now, with the variational formulae obtained in the previous section we will prove the Theorem \ref{theorem-B}. The proof is motivated by the results related to the volume comparison to the scalar curvature and $Q$-curvature contained in \cite{yuan2016volume} and \cite{lin2021deformations}, respectively. In this way, we will not  provide all details of the proof, which can be found in these references.

We remark that it is well known, see for example \cite[Lemma 4.57]{besse} or \cite{MR3524218}, that  if $(M^{n},  g_0)$ is a closed Einstein manifold, but not the standard sphere, then  we have the direct sum decomposition
$$S_2(M) = \text{Im}\ \delta^* \oplus (C^{\infty}(M)\cdot g_0)  \oplus S_{2,g_0}^{_{TT}}(M).$$
For the standard sphere $\mathbb S^n(\frac{1}{\sqrt{\lambda}})$ of radius $\frac{1}{\sqrt{\lambda}}$, the same result is true if the factor $C^{\infty}(M)$ is replaced by the $L^2$-orthogonal space to the first order spherical harmonics, i.e., by the space $E_\lambda^\perp$, where $E_\lambda:=\{u\in C^\infty(\mathbb S^n(\frac{1}{\sqrt{\lambda}})):\Delta_{\mathbb S^n(\frac{1}{\sqrt{\lambda}})} u+n\lambda u=0\}$.

A {\it local slice} $\mathcal S_{g_0}$ is a set of equivalence classes of metrics near $g_0$ modulo diffeomorphisms. For any closed Einstein manifold $(M,g_0)$, there exists a local slice $\mathcal S_{g_0}$ through $g_0$ in the space of all Riemannian metrics $\mathcal M$. This means that for a fixed real number $p>n$, there exists $\varepsilon>0$ such that  for any metric $g\in\mathcal M$ with $\|g-g_0\|_{W^{2,p}(M,g_0)}<\varepsilon$, there exists a diffeomorphism $\varphi$ with $\varphi^*g\in\mathcal S_{g_0}$, see \cite[Theorem 5.6]{yuan2016volume} for details.

The next result is fundamental to the proof of the volume comparison result, which is a slight modification of   \cite[Proposition 5.8]{lin2016deformations} and \cite[Proposition 5.7]{lin2021deformations}. We prove it by completeness.

\begin{proposition}\label{proposb}
Let $(M^{n},g_0)$ be a strictly stable  Einstein manifold  satisfying $Ric_{g_0} = (n-1)\lambda g_0,$ with $\lambda> 0$. Then there exists a local slice $\mathcal S_{g_0}$ through $g_0$ and a neighborhood $U_{g_0}$ of $ g_0$ in  $\mathcal{S}_{ g_0}$, such that any metric $g \in U_{g_0}$ satisfying
	$
	\mathcal{E}_{g_0}|_{\mathcal{S}_{ g_0}} (g) \geq \mathcal{E}_{g_0}|_{\mathcal{S}_{ g_0}}( g_0)
	$
	implies that $g = c^2  g_0$ for some constant $c > 0$.
\end{proposition}
\begin{proof}

Using the Ebin-Palais Slice Theorem (see \cite[Theorem 5.6]{lin2016deformations} and \cite[Theorem 5.3]{lin2021deformations}) there exists a local slice $\mathcal S_{g_0}$  through $g_0$, such that
$S_2(M)=T_{g_0}\mathcal S_{g_0}\oplus(T_{g_0}\mathcal S_{g_0})^\perp,$
where
\begin{itemize}
    \item $T_{g_0}\mathcal S_{g_0}:=S_{2,g_0}^{TT}(M)\oplus(C^\infty(M)\cdot g_0)$ and $(T_{g_0}\mathcal S_{g_0})^\perp=\{\delta^*X:\langle X,\nabla_{g_0}u\rangle_{L^2}=0,\forall u\in C^\infty(M)\}$
    when $(M,g_0)$ is not isometric to the round sphere, up to rescaling.
    \item $T_{g_0}\mathcal S_{g_0}:=S_{2,g_0}^{TT}(M)\oplus(E_\lambda^\perp\cdot g_0)$, $(T_{g_0}\mathcal S_{g_0})^\perp=\{\delta^*X:\langle X,\nabla_{g_0}u\rangle_{L^2}=0,\forall u\in E_\lambda^\perp\}$, when $(M,g_0)$ is isometric to the round sphere with radius $\frac{1}{\sqrt{\lambda}}$. Here, $E_\lambda=\{u\in C^\infty(\mathbb S^n(1/\sqrt{\lambda})):\Delta u+n\lambda u=0\}$ is the space of first eigenfunctions for the spherical metric.
\end{itemize}
 From Proposition \ref{propo004} we conclude that $g_0$ is a critical point of $\left.\mathcal E_{g_0}\right|_{\mathcal S_{g_0}}$ with $D^2\mathcal E_{g_0}(h,h)\leq 0$ for all $h\in T_{g_0}\mathcal S_{g_0}$. Since $g_0$ is a strictly stable Einstein metric, by \cite[Corollary 3.4]{MR558319} we conclude that $g_0$ is rigid, in the sense that there is a neighborhood $U_{g_0}$ of $g_0$ such that an Einstein metric $g\in U_{g_0}$ is of constant sectional curvature, which implies that $g=cg_0$ for some positive constant $c>0$. Define $$\mathcal Q_{g_0}:=\{g\in U_{g_0}\cap\mathcal S_{g_0}:g\mbox{ is Einstein}\}=\{g\in U_{g_0}\cap\mathcal S_{g_0}:g=cg_0,\mbox{ with }c>0\mbox{ constant}\}.$$
 In particular, the tangent space of $\mathcal Q_{g_0}$ at $g_0$ is given by $T_{g_0}\mathcal Q_{g_0}=\mathbb R g_0$ and its $L^2$-orthogonal complement $\mathcal C_{g_0}$ in $T_{g_0}\mathcal S_{g_0}$ is given by
$\mathcal C_{g_0}:=\left\{ h\in T_{g_0}\mathcal S_{g_0}: \int_Mtr_{g_0}hdv_{g_0}=0 \right\},$
since a 2-tensor $h\in T_{g_0}\mathcal S_{g_0}$ can be written as $h=\mathring h+\frac{1}{n}(tr_{g_0}h)g_0$. By Proposition \ref{propo004} we obtain that $D^2\mathcal E_{g_0}(h,h)<0$  for all $h\in\mathcal C_{g_0}$.

 Using a similar argument as in \cite[Proposition 5.8]{lin2016deformations} and \cite[Proposition 5.7]{lin2021deformations}, we define a weak Riemannian structure\footnote{The term weak  is due  the fact that \eqref{scalarprod} defines in each tangent space a topology weaker  than the current.}  on the local slice $\mathcal{S}_{g_0}.$
\begin{equation}\label{scalarprod}
( h, h)_{\bar{g}}:=\int_{M}\left[\langle h, h\rangle_{\bar{g}}+\left\langle\nabla_{\bar{g}} h, \nabla_{\bar{g}} h\right\rangle_{\bar{g}}\right] d v_{\bar{g}}=\int_{M}\left\langle\left(1-\Delta_{\bar{g}}\right) h, h\right\rangle_{\bar{g}} d v_{\bar{g}}
\end{equation}on $\mathcal{S}_{g_0}.$ 
We observe that it has a smooth connection by \cite{Ebin}. Define a vector field $Z$  on $\mathcal{S}_{g_0}$ as
$$
Z\left(\bar{g}\right):=\mathrm{V}\left(\bar{g}\right)^{\frac{4}{n}}\left(\Lambda_{\bar{g}}^{*}\left(f_{\bar{g}}\right)+\frac{2}{n} \bar{g} \mathrm{~V}\left(\bar{g}\right)^{-\frac{n+4}{n}} \mathcal{E}_{g_0}\left(\bar{g}\right)\right),
$$
where  $f_{\bar{g}}$ is a smooth positive function on $M$ satisfying $d v_{g_0}=f_{\bar{g}} d v_{\bar{g}}$. Since $g_0$  is Einstein,  we have $Z(g_0)=0$.  Note that in the scalar product \eqref{scalarprod} the gradient of $\left.\mathcal{E}_{g_0}\right|_{\mathcal{S}_{g_0}}$ at $\bar g$ is given by
$
Y\left(\bar{g}\right)=P_{\bar{g}}\left(\left(1-\Delta_{\bar{g}}\right)^{-1}\left(Z\left(\bar{g}\right)\right)\right),
$
where $P_{\bar{g}}:S_2(M)\rightarrow T_{\bar{g}}\mathcal S_{g_0}$ is the orthogonal projection to $T_{\bar{g}} \mathcal{S}_{g_0}$. This means that for any $h\in T_{\bar g}\mathcal S_{g_0}$ it holds
$D\left(\left.\mathcal E_{g_0}\right|_{\mathcal S_{g_0}}\right)_{\bar g}(h)=(h,Y(\bar g))_{\bar g}.$
This implies that for any $h=\mathring h + \frac{1}{n} (tr_{g_0}h)g_0  \in \mathcal{C}_{g_0}$ we have
$D^2\left(\left.\mathcal E_{g_0}\right|_{\mathcal S_{g_0}}\right)_{\bar g}(h,h)=(h,DY_{\bar g}(h))_{\bar g}.$
Since $\left.D^{2} \mathcal{E}_{g_0}\right|_{\mathcal{S}_{g_0}}(h,h)<0$ on $\mathcal{C}_{g_0}$, we have that $DY_{g_0}:\mathcal C_{g_0}\rightarrow\mathcal C_{g_0}$ is an isomorphism.

Using \cite[Lemma 5]{MR0380907} we can find a neighborhood $U_{g_0}\subset \mathcal S_{g_0}$ such that any metric $g\in U_{g_0}$ satisfying $\mathcal E_{g_0}(g)\geq \mathcal E_{g_0}(g_0)$ implies that $g\in \mathcal Q_{g_0}$. From the definition of $\mathcal Q_{g_0}$ we conclude that $g=c^2g_0$ for some positive constant $c>0$.

\end{proof}

Now we are ready to present a proof of Theorem \ref{theorem-B}.

\begin{proof}[Proof of Theorem \ref{theorem-B}]
By Ebin-Palais Slice Theorem \cite[Theorem 5.6]{lin2016deformations} we can find a sufficiently small  constant $\varepsilon_0>0$ such that for any metric $g$ satisfying
$ \| g - g_0\|_{C^2} < \varepsilon_0$,
 there exists a diffeomorphism $\varphi$ such that $\varphi^{*} g\in U_{g_0} \subseteq \mathcal{S}_{g_0}$, where $U_{ g_0}$ is given by Proposition \ref{proposb}.

Assume that $\sigma_2(g)\geq \sigma_2(g_0)$, $\|g - g_0\|_{C^2}<\varepsilon_0$ and
\begin{align}\label{volcomp}
	 V(g)\geq  V(g_0)
\end{align}
for some Riemannian metric $g$ on $M$. Since  there exists a diffeomorphism $\varphi$ such that $\varphi^{*}g\in U_{{g}_0}$ and 
$
\mathcal{E}_{g_0}|_{\mathcal{S}_{g_0}}(\varphi^* g)=\mathcal{E}_{g_0}|_{\mathcal{S}_{g_0}}(g)\geq  \mathcal{E}_{g_0}|_{\mathcal{S}_{g_0}}(g_0),
$
where we used \eqref{volcomp} and that  $\sigma_2(g_0)$ is constant.
By  Proposition \ref{proposb}, we have that $\varphi^{*}g=c^2g_0$ for some constant $c>0.$ Observe that  \eqref{volcomp} becomes
$$ V(g)= V(\varphi^*g)=c^n V(g_0) \geq  V(g_0).$$
Thus $c\geq 1$. On the other hand, 
$
	\sigma_2(g_0)= \sigma_2(\varphi^*g_0) \leq \sigma_2(\varphi^*g)=c^{-4}\sigma_2( g_0),
$
which implies that $c \leq 1.$
Hence, $\varphi^* g = g_0$ and the result follows. 
\end{proof}

\section{Variational Characterization of Critical Metrics of the Volume Functional}\label{sec003}

In Section \ref{closed section} we have investigated critical points of the volume functional in a manifold without boundary. In this section we study variational properties of the volume functional constrainted to the space of metrics of constant $\sigma_2$-curvature with a prescribed boundary metric. 

 Let $M^{n}$ be a connected, compact manifold of dimension $n\geq 3$ with smooth nonempty boundary $\partial M$ and a fixed boundary metric $\gamma$. Let $\mathcal M_\gamma\subset\mathcal M$ be the space of metrics on $M$ with induced metric on $\partial M$ given by $\gamma$. Let $K$ be a constant and $\mathcal M_\gamma^K$ be the space of metrics $g\in\mathcal M_\gamma$ which have constant $\sigma_2$-curvature $K$. Let $\mathcal S^{k,2}(M)$ be the space of $W^{k,2}$ symmetric 2-tensors on $M$, with $k>n/2+2$. Thus each $h\in\mathcal S^{k,2}(M)$ is $C^{2,\alpha}$ up to the boundary. By \cite[Lemma 1]{MR0380907} we get that $\sigma_2:\mathcal M_\gamma\rightarrow W^{k-2,2}(M)$ is smooth, where $W^{k-2,2}(M)$ is the space of $W^{k-2,2}$ functions on $M$.  Let $C_0^\infty(M)$ be the space of smooth function which vanishes in $\partial M$.

 \begin{lemma}\label{lemma-auxiliar-teorema-c}
     Let $g_0\in \mathcal M_{\gamma}^K$ be a 2-admissible metric on $M$ such that the first Dirichlet eigenvalue of the operator $-\mathcal T_{g_0}$ is positive.  Then there exists a neighborhood $U\subset\mathcal M$ of $g_0$ and an unique smooth function $\Phi:U\rightarrow C_0^\infty(M)$ such that for every $g\in U$ it holds $e^{2\Phi(g)} g\in\mathcal M_\gamma^K$.
 \end{lemma}
 \begin{proof}
 Consider the smooth map $\mathcal F:\mathcal M\times C_0^\infty(M)\rightarrow C^\infty(M)$ given by
 $$\mathcal F( g,u)=\sigma_2(e^{2u}g).$$
 Differentiating with respect to $u$ we get
 $$D_2\mathcal F{(g_0,0)}(v)=2\Lambda_{g_0}(vg_0)=-2\mathcal T_{g_0}(v),$$
 where we used Lemma \ref{Lem008} and $\mathcal T_{g_0}$ is defined in \eqref{eq049}. Let $f\in C^\infty(M)$ and consider the following boundary value problem
\begin{equation}
\left\{\begin{array}{rccl}
-\mathcal T_{g_0}(u) & = &f &  \text{ in $M$}\label{eq1009} \\                                 
u & = & 0 & \text{ on }\partial M.
\end{array}\right.
\end{equation}
By hypothesis $g_0$ is an admissible metric, which implies that (\ref{eq1009}) is an elliptic equation and has a unique solution by the Fredholm alternative, see \cite[Theorem 2.2.4]{MR0463624}. Thus, the operator $D_2\mathcal F{(g_0,0)}:C_0^\infty(M)\rightarrow C^\infty(M)$ is an isomorphism.  The Implicit Function Theorem for Banach spaces \cite[Theorem 5.9]{MR1666820} implies that there exists a neighborhood $U\subset\mathcal M$ of $g_0$ and  an unique smooth function $\Phi:U\rightarrow C_0^\infty(M)$ such that $\mathcal F(g,\Phi(g))=K$, for all $g\in U$, i.e. $e^{2\Phi(g)}g\in\mathcal M_\gamma^K$, for every $g\in U\subset\mathcal M$.
 \end{proof}

Next, we consider the volume functional $V:\mathcal{M}_\gamma\to\mathbb{R}$, whose  first variation (see Proposition 1.186 of \cite{besse}) is given by
\begin{equation}\label{eqvol}
DV_g(h)=\frac{1}{2}\displaystyle\int tr_ghdv_g.
\end{equation}
We are interested in critical points of $V$ restricted to $\mathcal{M}^K_{\gamma}$.
\begin{theorem}[Theorem \ref{theorem-C}] \label{main}
Let $g\in \mathcal M_{\gamma}^K$ be a 2-admissible metric such that  the first Dirichlet eigenvalue of $-\mathcal T_{g}$ is positive. Then, $g$ is a critical point of the volume functional in $\mathcal{M}_{\gamma}^K$ if and only if there exists a smooth function $f$ on $M$ such that
\begin{equation}\label{eq010}
\left\{
\begin{array}{rcl}
\Lambda_g^*(f) =  g & \text{ in } & M \\
f  =  0 & \text{ on } & \partial M.
 \end{array}
 \right.
\end{equation}
\end{theorem}

\begin{proof}
Suppose that $g$ is a critical point of $V$ in $\mathcal{M}_{\gamma}^K.$ Since $g$ is a 2-admissible metric and the first eigenvalue of $-\mathcal T_g$ is positive, it follows by the Fredholm alternative \cite[Theorem 2.2.4]{MR0463624} that there exists an unique function $f$ on $M$ satisfying the following equation
\begin{equation}\label{systemb}
\left\{
\begin{array}{rclcl}
\mathcal T_g(f) & = &\displaystyle -n &\ \text{ in }& M \\
  f&=&0& \text{ on }& \partial M.
  \end{array}
\right.
\end{equation}

We will prove that $f$ satisfies the equation (\ref{eq010}). Let $h$ be a smooth symmetric 2-tensor such that $h|_{T\partial M}\equiv 0$.  For small $|t|$ we have that $g(t)=g+th$ is a smooth metric in $\mathcal{M}_{\gamma}$. Define $u(t)=\Phi(g+th)$, where $\Phi$ is given by Lemma \ref{lemma-auxiliar-teorema-c}. By the uniqueness we obtain that $u(0)\equiv 0$ and $u'\equiv 0$ in $\partial M$, since $\sigma_2(g)=K$ and $u\equiv 0$ in $\partial M$. Thus, if $\tilde g(t)=e^{2u(t)}g(t)$, then 
\begin{equation}\label{eq019}
\sigma_2(\tilde g(t))=K.    
\end{equation}
 Note that 
$\tilde{g}'(0)=2u'(0)g+h$. 
Since $\tilde g(t)$ and $g(t)$ are conformal metrics, using \eqref{eq019}, their $\sigma_2$-curvature are related by 
\begin{equation}\label{eq024}
   -\mathcal T_{g(t)}(u)+\left(2u(t)+\frac{1}{2}\right)\sigma_2(g(t))-\frac{1}{2}Ke^{4u(t)}+\mathcal I_{g(t)}(u(t))=0,
\end{equation}
where $\mathcal I_{g}$ is given by \eqref{eq015}. Taking the derivative of \eqref{eq024} with respect to $t$ and using that $u(0)=0$ and $\mathcal I_g$ is quadratic in $u$ we obtain that $u'(0)$ satisfies 
\begin{equation}\label{eq027}
\left\{\begin{array}{rcll}
         \displaystyle \mathcal T_{g_0}(u'(0))  & = &\displaystyle \frac{1}{2}\Lambda_g(h)& \mbox{ in }M\\
         u'(0) & = & 0 & \mbox{ on }\partial M.
    \end{array}\right.
\end{equation}
Hence, by equation (\ref{systemb}) and using integration by parts we obtain
\begin{equation}\label{eq014}
\begin{array}{rcl}
n\displaystyle\int_{M}u'(0)dv_g&=&\displaystyle-\int_{M}u'(0)\mathcal T_g(f)dv_g=- \displaystyle\int_{M}f\mathcal T_g(u'(0))dv_g
\\
& = &- \displaystyle\frac{1}{2}\int_Mf\Lambda_g(h)dv_g=- \displaystyle\frac{1}{2}\int_M\langle h,\Lambda^*_g(f)\rangle dv_g,
\end{array}
\end{equation}
where in the third equality we have used \eqref{eq027}. Since $g$ is a critical point of $V$ in $\mathcal{M}^K_{\gamma}$, by (\ref{eqvol}), we have
$\displaystyle\int_{M}\left(2nu'(0)+tr_gh\right)dv_g=0.$ Using this and (\ref{eq014}) we obtain
$$
\displaystyle\int_{M}\left\langle h,\Lambda^*_g(f)-g\right\rangle dv_g=\displaystyle\int_{M}\left(\left\langle h,\Lambda^*_g(f)\right\rangle -tr_gh\right)dv_g=0.
$$
Since $h$ is any 2-tensor, we conclude that $\Lambda^*_g(f)=g.$

Now, suppose that $f$ satisfies the equation \eqref{eq010}. Let $h$ be a smooth symmetric 2-tensor in the tangent space of $g$ in $\mathcal M_\gamma^K$. This implies that $h|_{T\partial M}\equiv 0$ and $\Lambda_g(h)=0$. Therefore, using integration by parts, we obtain
\begin{equation}
    \begin{array}{rcl}\label{eq022}
         0 & = & \displaystyle\int_Mf\Lambda_g(h)dv_g=\int_M\langle h,\Lambda^*_g(f)\rangle dv_g \\
         &  &\displaystyle-\frac{1}{2(n-2)^2}\int_{\partial M} \left(\langle\nabla_\nu(fRic_g)\right. -\delta(fRic_g)(\nu)g,h \rangle +2\langle h(\nu),\delta(fRic_g) \rangle
         \\
         & &\left.-\dfrac{n}{2(n-1)}\left(\nabla_\nu(fR_g)tr_gh-\langle h(\nu),\nabla(fR_g)\rangle  \right)\right)d\sigma_g
         \\
         &=&\displaystyle\int_Mtr_ghdv_g=2DV(h),
    \end{array}
\end{equation}
since  $h|_{T\partial M}\equiv 0$ and $f\equiv 0$ on $\partial M$ implies that
$$\left\langle\nabla_\nu(fRic_g) -\delta(fRic_g)(\nu)g,h \right\rangle +2\langle h(\nu),\delta(fRic_g) \rangle=\nabla_\nu(fR_g)tr_gh-\langle h(\nu),\nabla(fR_g)\rangle=0,$$
on the boundary. Hence $g$ is a critical point of the volume functional in $\mathcal M_\gamma^K$.
\end{proof}

 Theorem \ref{main} gives conditions for a constant $\sigma_2$-curvature metric $g$ to be a critical point of the volume functional in $\mathcal{M}_{\gamma}^K,$ which is equivalent to the existence of  a function $f$  satisfying   $\Lambda_g^*(f) =  g$ in $M$ with  $f  =  0$ on $\partial M$.

It would be interesting to know the first variation of the volume functional in $\mathcal{M}_{\gamma}^K$ if  $\Lambda_g^*(f) =  g$ in $M$, but $f$ is not necessarily null on $\partial M.$ We could not find the expression in the general case, but restrict to a conformal class we have the following.

In 2009, Chen \cite{MR2570314} introduced the $H_k$-curvature of the boundary of a Riemannian manifold $(M^{n},g)$, which for $k=2$ one has
$$H_2=\frac{R_{\overline{g}}}{2(n-2)}H-\frac{n-1}{6}H^3-\frac{1}{n-2}\langle A_0,A_{\overline g}\rangle+\frac{1}{2(n-2)}H|A_0|^2,$$
where $A_0=A-H\overline g$, $\overline g$ is the induced metric on the boundary $\partial M$, $A$ is the second fundamental form, $H=\frac{1}{n-1}tr_{\overline g} A$ is the mean curvature of the boundary and
$A_{\overline g}$ is the Schouten tensor of the boundary $(\partial M,\overline g)$.
The definition for general $k$ can be found in \cite{MR3853045,MR2570314}.

By \eqref{eq049} we obtain
\begin{equation}\label{eq011}
\left.\frac{\partial}{\partial t}\right|_{t=0}\sigma_2(e^{2tu}g)=-2\mathcal T_{g_0}(u)
\end{equation}
and it is well known that
\begin{equation}\label{eq012}
\left.\frac{\partial}{\partial t}\right|_{t=0}H_2(e^{2tu}g)=-3uH_2(g)+T_1(\eta,\nabla_gu)-div_{\overline g}(H(g)\overline{\nabla}u),
\end{equation}
where $T_1$ is defined in \eqref{eq018}. See \cite{MR3853045,MR2570314,MR1738176} for details.

\begin{proposition}\label{VH}
Let $g\in\mathcal M_{\gamma}^K$ be a smooth metric. Let $f$ be a smooth function on $M$ such that
$\Lambda_g^*(f)=g$ on $M.$ 
Consider $\{e^{2tu}g\}$ a smooth path of conformal metrics in $\mathcal M_{\gamma}^K$. Then
$$\left.\frac{d}{dt}\right|_{t=0}V(e^{2tu}g)=\int_{\partial M}fH_2'(0)d\sigma_g,$$
where $H_2(t)$ is the $H_2$-curvature of $\partial M$ in $(M,e^{2tu}g)$ with respect to the unit outward
pointing normal vector.
\end{proposition}
\begin{proof}
Note that, by \eqref{eq022} we get
$$
\begin{array}{rcl}
\displaystyle\int_Mf\Lambda_g(h)dv_g
&=&\displaystyle\int_Mtr_ghdv_g-\int_{\partial M}fT_1(\eta,\nabla_g u)hd\sigma_g
\end{array}$$
By \eqref{eq012}, if $u\equiv 0$ in $\partial M$, we get
$\left.\frac{\partial }{\partial t}\right|_{t=0}H_2(e^{2tu}g)=T_1(\eta,\nabla_g u).$

Thus, by \eqref{eqvol} we conclude our result.
\end{proof}

\section{Critical Metrics in Space Forms}\label{space form}
Before presenting rigidity results in space forms, we give  some examples of functions satisfying equation \eqref{eq010}, such functions will be called  potential functions.

\begin{definition}
	Given a connected compact  manifold $M$ with smooth connected boundary $\partial M$, we say a metric $g$ on $M$ is a {\bf $\sigma_2$-critical metric} if there exists a potential function $f$ such that
	\begin{equation}\label{eq023}
	\left\{
	\begin{array}{rcl}
	\Lambda_g^*(f) =  g & \text{ in } & M \\
	f  =  0 & \text{ on } & \partial M.                                                
	\end{array}
	\right.
	\end{equation}
\end{definition}

We will assume that $f^{-1}(0)=\partial M$, which implies that $f$ does not change sign. In the case that $g$ satisfies  $\Lambda_{g}^*(f) = \kappa g$ in $M$, where $\kappa$ is a real constant, we also can assume that $f>0$  in $M \backslash \partial M$, since for the case $f< 0$, we only need to replace $(f,\kappa)$ by $(-f,-\kappa).$ If $\kappa = 0$ such metrics are called  \textbf{$\sigma_2$-singular metric} (see \cite{MR3828913}).

First we observe that if $g$ is an Einstein metric we can rewrite $ \Lambda_g^*(f)$ as
\begin{equation}\label{eq029}
    \Lambda_g^*(f)=\frac{R_g}{4n(n-1)}\left(\nabla^2f-(\Delta_gf)g-\frac{R_g}{n}fg\right).
\end{equation}

Now, we shall discuss the volume functional on domains in space forms. First we show two examples of $\sigma_2$-critical metrics.

\begin{example}\label{geodsphere}
Let $\Omega$ be a geodesic ball in the round sphere $\mathbb S^{n}$ with center $p$ and radius $R<\pi/2$. Consider the function $f:\Omega\to\mathbb R$ given by  $$f=\frac{4}{n-1}\left(\frac{\cos r}{\cos R}-1\right).$$
It is not difficult to see that $f$  satisfies \eqref{eq010}.
\end{example}

\begin{example}\label{geodhiperb}
Consider the hyperboloid model for hyperbolic space 
 $
 \mathbb H^{n}=\{x\in \mathbb{R}^{n,1}; t>0,\langle x,x \rangle_L=-1\},$ where $\mathbb{R}^{n,1}$ is the Minkowski space with the standard flat metric.
  Let $p=(1,0,\cdots,0)\in\mathbb{H}^{n},$ and $\Omega$ be a geodesic ball in the hyperboloid model for hyperbolic space $\mathbb H^{n}$ with center $p$ and radius $R$. It is not difficult to see that the function 	$$f(t,x_1,\cdots,x_{n})=\dfrac{4}{n-1}\left(\dfrac{\cosh r}{\cosh R}-1\right),$$ satisfies \eqref{eq010}, where $r=\cosh^{-1} t$ is the geodesic distance from $(t,x_1,\cdots,x_{n})$ to $p.$ 
\end{example}

Next we show that the only domains in the space forms $\mathbb{H}^{n}$ or $\mathbb{S}^{n}$, on which the canonical metrics are critical points, are geodesic balls.

\begin{theorem} \label{classi}Let $(\Omega,g)$ be 
a bounded connected domain with smooth boundary $\partial \Omega$ with a fixed boundary 
metric $\gamma=g|_{T\partial \Omega}$ in a simply connected space form in $\mathbb{H}^{n}$ or $\mathbb{S}^{n}$. In the case of $\Omega\subset\mathbb{S}^{n},$ we also assume that $V(\Omega)<\frac{1}{2}V(\mathbb{S}^{n}).$ Suppose that $g$ is a critical point  of the volume functional $V$ in $M_{\gamma}^K,$ where $K=\frac{n(n-1)}{8}.$
Then, the corresponding space form metric is a critical metric on $\Omega$ if and only if $\Omega$ is a geodesic ball.
\end{theorem}

\begin{proof}
We observe that if $\Omega\subset\mathbb{H}^{n}$ 
or $\Omega\subset\mathbb{S}^{n},$
then $R_g=-n(n-1)$  or 
 $R_g=n(n-1),$ respectively. In the special case that the volume of $\Omega\subset \mathbb S^n$ is less than the volume of a hemisphere,  $\Omega$ must be strictly contained in the upper hemisphere and the first eigenvalue of $\mathcal{T}_g$ is positive by the  Faber-Krahn inequality \cite{sperner1973symmetrisierung}. Under our assumption $g$ is a critical point of the volume functional $V|_{\mathcal{M}_{\gamma}^K}$ if and only if there exists a smooth function $f$ on $M$ such that $f  =  0$  on $\partial \Omega$ and

 	\begin{equation}
\left\{
\begin{array}{rcll}
\displaystyle R_g\nabla^2f+\frac{R_g^2}{n(n-1)}fg &=& -4n g & \text{ in }  \Omega \\
f &=&0   & \text{ on }  \partial \Omega. 
\end{array}
\right.
\end{equation} The theorem follows by a slightly modification of the proof of Theorem 6 in  \cite{MR2546025} we obtain the desired result.
 \end{proof}

The following is an immediate consequence of the theorem.

\begin{coring} Let $\Omega$ be a connected domain with compact closure in $\mathbb{H}^{n}$ or $\mathbb{S}^{n}$ and with a smooth (possibly disconnected) boundary $\partial \Omega.$  Let $g$ be the standard metric in $\Omega.$ If $\Omega\subset \mathbb{S}^{n},$ we also assume that $V(\Omega)<\frac{1}{2}V(\mathbb{S}^{n}).$ Then, $\Omega$ is a geodesic ball if and only if 
$$\int_{\partial \Omega}H_2'(0)d\sigma_g=0,$$
for any smooth variation 
$\{e^{2tu}g\}$ of $g$ in $\mathcal M_{\gamma}^K$. 
\end{coring}

\begin{proof}

Since $\Omega$  is a connected domain with compact closure in $\mathbb{H}^{n}$ or $\mathbb{S}^{n},$ 
 the constant function $f=-4/(n-1)$ satisfies $\Lambda_g^*(f)=g.$ Then Proposition \ref{VH} implies 
$$
\left.\frac{d}{dt}\right|_{t=0}V(e^{2tu}g)=-\frac{4}{n-1}\int_{\partial \Omega}H_2'(0)d\sigma_g.
$$
By Theorem \ref{classi}, $g$ is a critical point of the volume functional $V$ in $\mathcal{M}_{\gamma}^K,$ and so 
$$\int_{\partial \Omega}H_2'(0)d\sigma_g=0,$$
for any smooth conformal variation  $\{e^{2tu}g\}$ of $g$ in $\mathcal{M}_{\gamma}^K.$ This completes the proof. 
\end{proof}

In particular, motivated by results of \cite{miao2011einstein} we have

\begin{proposition}
Suppose $(\Omega^{n},g)$ is a connected, compact, Einstein manifold with a smooth boundary $\partial \Omega.$ If there is a function $f$ on $\Omega$ such that $f=0$ on $\partial \Omega$ and $\Lambda^*_g(f)=g$
in $\Omega,$ then $(\Omega,g)$ is isometric to a geodesic ball in a simply connected space form $\mathbb{H}^n$ or $\mathbb{S}^n.$
\end{proposition}

\begin{proof}
We observe that if $g$ is an Einstein metric then by equation \eqref{eq029}, we obtain that $$\Lambda^*_g(f)=\dfrac{R_g}{4n(n-1)}L_g^*(f),$$
where $L_g^*$ is the $L^2$-formal adjoint of linearization of scalar curvature (see \cite{MR0380907}).
Note that $R_g=n(n-1)k,$ where $k=1$ or $k=-1.$ Thus, $\Lambda^*_g(f)=\dfrac{k}{4}L_g^*(f).$ Now, using the same idea that in Lemma 2.1, Lemma 2.2 and Theorem 2.1 in \cite{miao2011einstein} we conclude this result.
\end{proof}

\begin{remark}
Given a positive constant $k$ consider a geodesic ball $\Omega_{-k}$ in the hyperbolic space $\mathbb H^n_{-k}$ with sectional curvature $-k$. Consider the product $M=\Omega_{-k}\times\mathbb S^n_k$ with the product Riemannian metric $g$. Here $\mathbb S^n_k$ is the round sphere with sectional curvature $k$. It is well known that the canonical metric in $\mathbb S^n_k$ is a critical point for the volume functional restricted to the space of Riemannian metric with constant scalar curvature (see \cite{MR3096517}, for instance). But, using \eqref{eq007} and Theorem \ref{main} we find that the canonical metric in $\mathbb S^n_k$ is a critical point for the volume functional restricted to the space of Riemannian metrics with constant $\sigma_2$-curvature. By the Example \ref{geodhiperb} and Theorem \ref{main} we have that the canonical metric in $\Omega_{-k}$ is a critical point of the volume functional in $\mathcal M_\gamma^{-k}$. 

Note that $(M,g)$ is a smooth compact manifold with nonempty boundary, locally conformally flat and non Einstein manifold, scalar curvature identically zero and $\sigma_2(g)=-\frac{n(n-1)^2}{(n-2)^2}k^2=:c<0$. Since the volume of $M$ is the product of the volume of $\Omega_{-k}$ and $\mathbb S^n_k$, we conclude that $g$ is a critical metric of the volume functional restricted to $\mathcal M_{\overline\gamma}^c$, where $\overline\gamma$ is the fixed metric in the boundary $\partial M$.
\end{remark}

\section{Second variational formula for the volume functional}\label{secondvariat}

In this section we will use the second derivative of the $\sigma_2$-curvature given by Proposition \ref{lem004} to find the second variation of the volume functional at a critical metric in $\mathcal M_\gamma^K$. Then we find a direction where this variation is strictly negative.

\begin{theorem}\label{teo003}
Let $(M,g)$ be a  connected compact Riemannian  manifold of dimension $n$ with a nonempty smooth boundary, such that the first Dirichlet eigenvalue of $-\mathcal T_g$ is positive. Suppose $g$ has constant sectional curvature $\kappa$ and that there is a smooth function $f$ on $M$ satisfying 
\begin{equation}\label{equa001}
\left\{
\begin{array}{rcl}
\Lambda_g^*(f) =  g & \text{ in } & M \\
f  =  0 & \text{ on } & \partial M. 
 \end{array}
 \right.
\end{equation}
Let $\gamma=g_{T\partial M}$ and let $K$ be the constant that equals the $\sigma_2$-curvature of $g$. Suppose $\{g(t)\}$ is a smooth path of metrics in $\mathcal M_\gamma^K$ with $g(0)=g$. Let $V(t)$ be the volume of $(M,g(t))$, then
\begin{align*}
     V''(0)            = & \displaystyle\int_M\left(\frac{1}{4}(tr_gh)^2+f\left(\frac{1}{8(n-2)^2}\left| \Delta h+\nabla^2tr_gh+2\delta^*\delta(h)+2\kappa ((tr_gh)g-h)\right|^2+\frac{\kappa}{16} |\nabla h|^2\right.\right.
     \\
          & \displaystyle +\frac{\kappa^2}{8}|h|^2-\frac{3}{8}k^2(tr_gh)^2-\frac{\kappa}{4}|\delta h|^2-\kappa  \left\langle h,\frac{5n^2-14n+14}{4(n-2)^2}\delta^*\delta h+\frac{n^2-3n+3}{4(n-2)^2}\nabla^2tr_gh\right\rangle
     \\
     & 
     -\dfrac{n}{8(n-1)(n-2)^2}\left(\Delta tr_gh-\delta^2h+\kappa(n-1)tr_gh\right)^2 +\frac{n^2+2n-2 }{4(n-2)^2}\kappa  \left|\delta h+\frac{1}{2}\nabla tr_gh\right|^2
      \\
      &   \displaystyle\left.\left. -\frac{n-1}{2(n-2)^2}\kappa \left( \delta\left(h\circ \left(\delta h+\frac{1}{2}\nabla tr_gh\right)\right)+\left\langle \delta h+\frac{1}{2}\nabla tr_gh, \frac{1}{2}\nabla tr_gh\right\rangle\right)\right)\right),
\end{align*}
    where $h=g'(0)$, $h'= g''(0)$ and all covariant derivative are with respect to $g$.
\end{theorem}
\begin{proof}
We first note that $V'(0)=0$ by \eqref{equa001} and  Theorem \ref{main}.  Note that 
\begin{equation}\label{eq058}
    V''(0)=\int_M\left(\frac{1}{4}(tr_gh)^2-\frac{1}{2}|h|^2+\frac{1}{2}tr_gh'\right)dv_g,
\end{equation}
where $h=g'(0)$ and $h'=g''(0)$. Our aim is to express the last integral in terms of $h$. Applying the fact that $g(t)$ has constant $\sigma_2$-curvature $K$ and Corollary \ref{cor003}, we have
\begin{equation}\label{eq052}
    0=\sigma_2''(0)=c_n^{-1}I+\Lambda_g(h'),
\end{equation}
where $c_n=2(n-2)^2$ and $I$ involves only $h$ and its derivatives. Integrating by parts, we have
\begin{align*}
      \displaystyle \int_Mf\Lambda_g(h') = & \displaystyle \int_M\langle h',\Lambda_g^*(f)\rangle  -c_n^{-1}\int_{\partial M}\left(\langle\nabla_\nu(fRic_g)-\delta(fRic_g)(\nu)g,h' \rangle +2\langle h'(\nu),\delta(fRic_g) \rangle\right. \\
       &\left.-\dfrac{n}{2(n-1)}\left(\nabla_\nu(fR_g)tr_gh'-\langle h'(\nu),\nabla(fR_g)\rangle  \right)\right).
\end{align*}
 By \eqref{equa001} we have $\langle h',\Lambda_g^*(f)\rangle=tr_gh'$. Using that $h|_{T\partial M}\equiv 0$ and $f\equiv 0$ on $\partial M$, as in \eqref{eq022}, and \eqref{eq052} we obtain that 
\begin{equation}\label{eq056}
    \int_Mtr_gh'=\int_Mf\Lambda_g(h')=-c_n^{-1}\int_MfIdv_g.
\end{equation}

Since $g$ has constant sectional curvature $\kappa$, then $R_{ijkl} =\kappa (g_{il}g_{jk}-g_{ik}g_{jl})$, $Ric_g=(n-1)\kappa g$, $R_g=n(n-1)\kappa$, $\sigma_2(g)=\frac{n(n-1)}{8}k^2$, $\Lambda_g^*(f)=\frac{\kappa}{4}\left(\nabla_g^2f-(\Delta_gf)g-(n-1)\kappa fg\right)$ and $tr_g\Lambda_g^*(f)=-\frac{n-1}{4}\kappa(\Delta_gf+n\kappa f)$. By \eqref{equa001} we obtain
\begin{equation}\label{eq055}
    \Delta_gf=-\frac{4n}{(n-1)\kappa}-n\kappa f\quad\mbox{ and }\quad \nabla^2f=\left(-\frac{4}{(n-1)\kappa}-\kappa f\right)g.
\end{equation}

 Integrating by parts and \eqref{eq055} implies that
 $$\int_Mf\Delta|h|^2dv_g=-\int_M|h|^2\left(n\kappa f+\frac{4n}{(n-1)\kappa}\right)dv_g-\int_{\partial M} |h|^2\partial_\nu f.$$
 
 Also, in an analogous way as in the proof of \cite[Theorem 9]{MR2546025},  we have 
 \begin{align*}
     \int_Mfg^{ij}g^{pq}g^{sl}\nabla_ph_{sj}\nabla_lh_{qi}  = &\displaystyle \int_Mg^{ij}g^{pq}g^{sl}(\nabla_p(fh_{sj}\nabla_lh_{qi})-\nabla _pfh_{sj}\nabla_lh_{qi}-fh_{sj}\nabla_p\nabla_lh_{qi})
      \\
       = &\displaystyle \int_Mg^{ij}g^{pq}g^{sl}(-\nabla_l(\nabla_pfh_{sj}h_{qi})+\nabla_l\nabla_pfh_{sj}h_{qi}+\nabla_pf\nabla_lh_{sj}h_{qi}
      \\
      &  -fh_{sj}\nabla_l\nabla_ph_{qi}+g^{mt}fh_{sj}R_{plqt}h_{mi}+g^{mt}fh_{sj}R_{plit}h_{qm}),
 \end{align*}
     where we used integration by parts and that $f\equiv 0$ on $\partial M$ to infer that the integral of the first term in the r.h.s of the first line is equal to zero. Also, we used in the second equality the Ricci identity.
     Using integration by parts, \eqref{eq054} and \eqref{eq055} we get
     \begin{align*}
         \int_M & fg^{ij}  g^{pq}g^{sl}  \nabla_ph_{sj}\nabla_lh_{qi}    =  \displaystyle -\int_{\partial M}g^{ij}g^{pq}\nabla_pfh_{\nu j}h_{qi}-\int_M\left(\kappa f+\frac{4}{(n-1)\kappa}\right)|h|^2
      \\
      &  \displaystyle +\int_M(f\langle h,\nabla \delta h\rangle-(n-1)\kappa f|h|^2+\kappa f(tr_gh)^2-\kappa f|h|^2
      \\
      & + g^{ij}g^{pq}g^{sl}(\nabla_p(f\nabla_lh_{sj}h_{qi})-f\nabla_p\nabla_lh_{sj}h_{qi}-f\nabla_lh_{sj}\nabla_ph_{qi}))
      \\
       = & \displaystyle \int_Mf(\kappa (tr_gh)^2+2\langle h,\delta^*\delta h\rangle-|\delta h|^2)-\int_M\left((n+1)\kappa f+\frac{4}{(n-1)\kappa}\right)|h|^2-\int_{\partial M}\nabla_\nu f|h|^2.
     \end{align*}
 Thus
\begin{equation}\label{eq057}
    \int_Mf(\Delta|h|^2-g^{ij}g^{pq}g^{sl}\nabla_ph_{sj}\nabla_lh_{qi})=\int_M\left(\kappa f-\frac{4}{\kappa}\right)|h|^2+\int_Mf(-\kappa (tr_gh)^2-2\langle h,\delta^*\delta h\rangle+|\delta h|^2).
\end{equation}
 
 Therefore, by Corollary \ref{cor003}, \eqref{eq058}, \eqref{eq056} and \eqref{eq057} we obtain the result.
\end{proof}

\begin{coring}\label{cor002}
Under the same assumption of Theorem \ref{teo003}, if $tr_gh=0$ and $div_gh=0$, then
 $$\begin{array}{rcl}
     V''(0)     & = & \displaystyle \int_Mf\left(\frac{1}{8(n-2)^2}\left| \Delta_gh-2\kappa h\right|^2 +\frac{\kappa^2}{8}|h|^2+\frac{\kappa}{16} |\nabla h|^2\right).
\end{array}$$
\end{coring}

Now  Theorem \ref{theorem-D} is a consequence of the following result.
\begin{theorem}\label{local_minimum}

Let $(\Omega,g)$ be a geodesic ball  with compact closure in $\mathbb{S}^n_+$ and $\mathbb{H}^n.$   Let $\gamma=g|_{T\partial \Omega}$ and $K$ be the constant equal the $\sigma_2$-curvature of $g.$  Suppose $\{g(t)\}$ is a smooth path of metrics in $\mathcal M_{
\gamma}$ with $g(0)=g$ and  $g^{\prime}(0)=h$.  Let $V(t)$ denote the volume of $(\Omega,g(t))$.
\begin{itemize}
\item[a)]If $\Omega\subset\mathbb{S}_+^{n}$, then $V''(0)>0$
for any $h$ satisfying
$
\operatorname{div}_{g} h=0$ and $\operatorname{tr}_{g} h=0.
$

\item[b)]  If $\Omega\subset \mathbb{H}^n,$ then for any point $p\in \Omega$ there exists a geodesic ball  with radius $\delta$ centered at $p$, denoted by $B(p,\delta),$ depending on $p$ and $\Omega$ such that $V''(0)>0$
for any $h$ which has compact support in $B(p,\delta)\subset (\Omega,g)$ and   satisfying
$
\operatorname{div}_{g} h=0 $ and $\operatorname{tr}_{g} h=0.
$
\end{itemize}
\end{theorem}

\begin{proof}
Suppose that $\Omega$ is a geodesic ball of $\mathbb{S}^n_+,$ then 
 $$\begin{array}{rcl}
     V''(0)     & = & \displaystyle \int_Mf\left(\frac{1}{8(n-2)^2}\left| \Delta_gh-2 h\right|^2 +\frac{1}{8}|h|^2+\frac{1}{16} |\nabla h|^2\right).
\end{array}$$
Since $f>0$ in $\Omega,$ $V''(0)>0.$

Next suppose $\Omega$ is a geodesic ball of $\mathbb{H}^n.$ Then 
 $$\begin{array}{rcl}
     V''(0)     & = & \displaystyle \int_Mf\left(\frac{1}{8(n-2)^2}\left| \Delta_gh+2 h\right|^2 +\frac{1}{8}|h|^2-\frac{1}{16} |\nabla h|^2\right).
\end{array}$$
Given $p \in \Omega$, we can find  $\delta>0$ such that $B(p,\delta) \subset \Omega$ satisfies 
$$
2\min _{B(p,\delta)} f \geq  f(p)\quad\mbox{ and } \quad \frac{1}{2}\max _{B(p,\delta)} f \leq f(p).
$$

Then for a compactly supported  tensor $h$ in $B(p,\delta)$, we have 

\begin{align*}
 V''(0)      \geq & ~ \displaystyle f(p)\int_{B(p,\delta)}\left(\frac{1}{16(n-2)^2}\left| \Delta_gh+2 h\right|^2 +\frac{1}{16}|h|^2-\frac{1}{8} |\nabla_gh|^2\right)\\
      = & ~ \displaystyle f(p)\int_{B(p,\delta)}\left(\frac{1}{16(n-2)^2}\left(| \Delta_gh|^2+4\langle h,\Delta_gh\rangle+4|h|^2\right)\right.\left.+\frac{1}{16}|h|^2+\frac{1}{8}\langle h,\Delta_gh\rangle\right).    
\end{align*}

Thus, 
$$\begin{array}{rcl}
     V''(0)     & \geq& \displaystyle \frac{f(p)}{16(n-2)^2}\Big(\lambda_1(B(p,\delta)^2-(4+2(n-2)^2)\lambda_1(B(p,\delta))+4+(n-2)^2\Big)\left(\int_{B(p,\delta)}|h|^2\right),
\end{array}$$
where $\lambda_1(B(p,\delta))$ is the first eigenvalue of the rough Laplacian on $B(p,\delta)$.

We conclude that  $V^{''}(0)>0$ as $\displaystyle\lim _{\delta \rightarrow 0} \lambda_{1}\left(B(p,\delta)\right)=+\infty$. Hence the result follows for a sufficiently smaller $\delta.$
\end{proof}

\begin{proof}[Proof of Theorem \ref{theorem-D}]

As in the proof of Theorem \ref{main},  we can find a family of metrics $\tilde g=e^{2u(t)}(g_0+th)$ with  constant $\sigma_2$-curvature equal to $K$, and if $v=\left.\frac{\partial u}{\partial t}\right|_{t=0}$, then we have 
\begin{equation}\label{eq005}
\left\{\begin{array}{rcll}
         \displaystyle \mathcal T_{g_0}(v)  & = &\displaystyle \frac{1}{2}\sigma''_2(0)& \mbox{ in }\Omega\\
         v & = & 0 & \mbox{ on }\partial \Omega.
    \end{array}\right.
\end{equation}
If $h$ satisfies $h|_{T\partial \Omega}=0,$ 
$
\operatorname{div}_{g} h=0$ and $\operatorname{tr}_{g} h=0,
$
then $\sigma''_2(0)=0.$ Since in a space form the first Dirichlet eigenvalue of $\mathcal T_g(u)$ is positive, the result follows from Theorem \ref{local_minimum} and the existence of  trace free and divergence free  symmetric 2-tensors with prescribed compact support on space forms (see Appendix \cite{MR2546025} or \cite{corvino2007existence}). 
\end{proof}

In the setting of critical metrics of the volume functional with constant scalar curvature, P. Miao and L.F. Tam \cite{MR2546025} proved  the existence of  deformations along which the volume of the standard metric is a strict local maximum. For that, they use a limit argument where the limit metric has zero scalar curvature (in fact in the Euclidean space). This fact was essential to show  further the  nonexistence  of a  global volume minimizer in 
$$
 \mathcal M_{\gamma}^0=\{g\in \mathcal M\mid R(g)=0\quad\mbox{and}\quad g|_{T\partial M}=\gamma\}.
$$
We observe that using the same technique an analogue of this fact in our context   is not possible  since   $\sigma_2$-critical metrics  do not have models in the Euclidean space (see Section \ref{space form}).  

\vspace{0.5cm}

\noindent{\bf Conflict of interest statement}

\vspace{0.5cm}

The authors declare that there is no conflict of interest regarding the publication of this article.

\bibliography{main-ref.bib}
\bibliographystyle{acm}

\end{document}